\documentclass{llncs}
\usepackage{amsmath,amssymb,amsfonts,mathrsfs,thm-restate,thmtools}
\usepackage{subcaption}
\usepackage{float}
\usepackage{caption}
\usepackage{stfloats}
\usepackage{graphicx}
\usepackage{appendix}
\usepackage{epsf}
\usepackage{epsfig}
\usepackage{epstopdf}
\usepackage{xspace,xcolor}
\usepackage{soul}
\usepackage{comment}
\usepackage{latexsym}
\usepackage{tikz}
\usepackage{framed}
\usepackage{todonotes}
\usepackage{lipsum}
\usepackage{setspace}
\usepackage{xcolor}
\usepackage{tabulary}
\usepackage{tabularx}
\usepackage{multirow}
\usepackage{booktabs}

\usepackage{algpseudocode}

\usepackage[linesnumbered,ruled,vlined]{algorithm2e}
\usepackage{mathtools}

\usepackage{caption}
\usepackage{subcaption}
\usepackage{tkz-tab}
\usetikzlibrary{arrows,positioning}
\usetikzlibrary{shapes,snakes}

\usepackage[bookmarks=false, colorlinks,
linkcolor=red!60!black, citecolor=green!60!black]{hyperref}
\usepackage{nameref}

\usepackage{enumerate}
\usepackage{caption}
\usepackage{subcaption}
\usepackage[compatibility=false]{caption}
\tikzstyle{vertex}=[circle, draw, inner sep=0pt, minimum size=4.5pt]

\title{On Locally Identifying Coloring of Cartesian Product and Tensor Product of Graphs}

\author{Sriram Bhyravarapu\inst{1}, 
Swati Kumari\inst{2}
  and  I. Vinod Reddy\inst{2}}

\institute{The Institute of Mathematical Sciences, HBNI, Chennai, India\\
\email{sriramb@imsc.res.in}\\ 
\and
Department of Computer Science and Engineering,  IIT Bhilai, India\\
\email{swatik@iitbhilai.ac.in, vinod@iitbhilai.ac.in}}

\begin{document}
	
	\pagestyle{plain}
	
	\maketitle
 \begin{abstract}
 For a positive integer $k$, a proper $k$-coloring of a graph $G$ is a mapping $f: V(G) \rightarrow \{1,2, \ldots, k\}$ such that $f(u) \neq f(v)$ for each edge $uv$ of $G$. The smallest  integer $k$ for which there is a proper $k$-coloring of $G$ is called the chromatic number of $G$, denoted by $\chi(G)$.
 A \emph{locally identifying coloring} (for short, lid-coloring) of a graph $G$
 is a proper $k$-coloring of $G$ such that every pair of adjacent vertices with distinct closed neighborhoods has  distinct set of colors in their closed neighborhoods. 
 The smallest integer $k$ such that $G$ has a lid-coloring with $k$ colors is called
 \emph{locally identifying chromatic number}
 (for short, \emph{lid-chromatic number}) of $G$, denoted by $\chi_{lid}(G)$.

This paper studies the lid-coloring of the Cartesian product and tensor product of two graphs.   
We prove that if $G$ and $H$ are two connected graphs having at least two vertices then (a) $\chi_{lid}(G \square H) \leq \chi(G) \chi(H)-1$ and (b) $\chi_{lid}(G \times H) \leq \chi(G) \chi(H)$. Here $G \square H$ and $G \times H$ denote the Cartesian and tensor products of $G$ and $H$ respectively. 
We determine the lid-chromatic number of $C_m \square P_n$, $C_m \square C_n$, $P_m \times P_n$, $C_m \times P_n$ and $C_m \times C_n$, where $C_m$ and $P_n$ denote a cycle and a path on $m$ and $n$ vertices respectively.

 \end{abstract}

\section{Introduction}\label{S:intro}

In this paper, we consider finite, undirected and simple graphs. For a graph $G=(V,E)$, the vertex set and edge set of $G$ are denoted by $V(G)$ and $E(G)$ respectively. The neighborhood $N(v)$ of a vertex $v$ in a graph $G$ is the set of vertices adjacent to $v$ in $G$
and $N[v]= N(v) \cup \{v\}$ denotes closed neighborhood of $v$. For a positive integer $k$, a $k$-coloring of a graph $G$ is a function $f: V(G) \rightarrow \{1,2,\ldots, k\}$.
A $k$-coloring of a graph $G$ is called \emph{proper $k$-coloring}, if $f(u) \neq f(v)$ for each edge $uv$ of $G$. The chromatic number $\chi(G)$ of a graph $G$ is the minimum $k$ for which there is a proper $k$-coloring of $G$. 
For a $k$-coloring $f$ of a graph $G$ and $X \subseteq V(G)$, we denote $f(X)=\{f(v)~|~ v \in X\}$.

Given a graph $G$ and a positive integer $k$, a proper $k$-coloring $f$ is called a \emph{locally identifying coloring} using $k$ colors (for short $k$-lid-coloring), 
if for every edge $uv \in E(G)$ with $N[u] \neq N[v]$, we have $f(N[u]) \neq f(N[v])$. The smallest integer $k$ such that there is a locally identifying coloring of $G$ using $k$ colors is called the locally identifying chromatic number of $G$ (or lid-chromatic number), denoted by $\chi_{lid}(G)$. In this paper, we consider only connected graphs since the lid-chromatic number of a graph $G$ is the maximum of the lid-chromatic numbers of its connected components.


The notion of locally identifying coloring was introduced by Esperet et al.~\cite{esperet2010locally}. The authors gave  bounds on lid-chromatic numbers for various families of graphs, such as planar graphs, interval graphs, split graphs, cographs and graphs with bounded maximum degree. They proved that the lid-chromatic number of a bipartite graph  is at most four and deciding whether 
 a bipartite graph  
is $3$ or $4$-lid-colorable 
is an 
 $\sf{NP}$-complete problem.   Foucaud et al.~\cite{foucaud2012locally} proved that any graph $G$ has a locally identifying coloring with at most $2 \Delta^2-3 \Delta+3$ colors, where $\Delta$ denotes the maximum degree of $G$.
 Goncalves et al.~\cite{gonccalves2013locally} showed that the lid-chromatic number for any graph class of bounded expansion is bounded. They also gave an upper bound on the lid-chromatic number of planar graphs. Martins and Sampaio~\cite{martins2018locally} gave linear time algorithms to calculate the lid-chromatic number for some classes of graphs having few $P_4$'s, such as  cographs, $P_4$-sparse graphs and $(q,q-4)$-graphs. We now formally introduce the definitions of Cartesian product and tensor product of graphs.


\begin{definition}[Cartesian product \cite{hammack2011handbook}]
    The Cartesian product $G\square H$ of graphs $G$ and $H$ is a graph such that $V(G\square H) = V(G) \times V(H)=\{(u,v)~|~ u \in V(G), v\in V(H)\}$,  and  $(u_1,v_1) (u_2,v_2) \in E(G \square H)$ if and only if either $u_1 = u_2$ and $v_1v_2 \in E(H)$ or $v_1 = v_2$ and $u_1u_2 \in E(G)$. 
\end{definition}
\begin{definition}[Tensor product \cite{hammack2011handbook}]
The tensor product $G \times H$ of graphs $G$ and $H$ is a graph such that $V(G \times H) = V(G) \times V(H)$ and $(u_1,v_1) (u_2,v_2) \in E(G \times H)$ if and only if $u_1u_2 \in E(G)$ and $v_1v_2 \in E(H)$.     
\end{definition}

Notice that both the Cartesian product and tensor product are commutative. That is, for any two graphs $G$ and $H$ we have $G \square H\cong H \square G$ and $G \times H \cong H \times G$~\cite{hammack2011handbook}.

Proper coloring has been well studied on various graph products~\cite{geller1975chromatic,sabidussi1957graphs,shitov2019counterexamples}.
It is known  that (a) $\chi(G\square H)=\max\{\chi(G), \chi(H)\}$~\cite{sabidussi1957graphs}, and 
 (b) $\chi(G\times H)\leq \min\{\chi(G), \chi(H)\}$~\cite{shitov2019counterexamples}.

In this paper, we investigate the lid-chromatic number of Cartesian product and tensor product of graphs. 
In Section \ref{sec:cartesian}, we prove that if $G$ and $H$ are two connected graphs having at least two vertices, then  $\chi_{lid}(G \square H) \leq \chi(G) \chi(H) - 1$.  We give exact values of lid-chromatic number of Cartesian product of (a) a cycle and a path, and  (b) two cycles.

In Section \ref{sec:tensor}, we prove that if $G$ and $H$ are two connected graphs having at least two vertices then  $\chi_{lid}(G \times H) \leq \chi(G) \chi(H)$.  We also give exact values of lid-chromatic number of tensor product of (a) two paths (b) a cycle and a path and (c) two cycles.

\section{Preliminaries} \label{sec-pre}
We use $[k]$ to denote the set $\{1, 2, \dots, k\}$.
For a positive integer $n$, we use $P_n$ to denote a path on $n$ vertices and $C_n$ to denote a cycle on $n$ vertices. Given a graph $G$ and a subset $X \subseteq V(G)$, we use $G[X]$ to denote the subgraph of $G$ induced by the vertices of $X$. For more details on graph theory, the reader can refer \cite{west2001introduction}.

\begin{lemma}[\cite{foucaud2012locally}]\label{lem:path}
For a positive integer $n$, where $n \geq 2$, we have 

$$
		\chi_{lid}(P_n) = 
			\begin{cases} 
				2 &\text{if $n =2$;} \\
				3 &\text{if $n=2p+1$ for some $p \in \mathbb{N}$; \;} \\
				4 &\text{if $n=2p+2$ for some $p \in \mathbb{N}$.} \\
			\end{cases}
$$	
\end{lemma}

\begin{lemma}[\cite{foucaud2012locally}]\label{lem-cycle}
For a positive integer $n$, where $n \geq 3$, we have 

$$
		\chi_{lid}(C_n) = 
			\begin{cases} 
				3 &\text{if $n=3$ or $n \equiv 0~(mod~4)$;} \\
				5 &\text{if $n=5$ or $7$;} \\
				4 &\text{otherwise.} \\

			\end{cases}
$$	
\end{lemma}

Next, we review some results from \cite{esperet2010locally} that are used to prove some of our results.
\begin{lemma}[\cite{esperet2010locally}]\label{lem-bipartite-triangle}
If a connected graph $G$ satisfies  $\chi_{lid}(G) \leq 3$, then $G$ is either a triangle or a bipartite graph.
\end{lemma}

\begin{theorem}[\cite{esperet2010locally}]
\label{lem-bipartite}
 If $G$ is a bipartite graph, then $\chi_{lid}(G) \leq 4$.
\end{theorem}

\begin{theorem}[\cite{esperet2010locally}] \label{lem-regular}
For $k \geq 4$, a $k$-regular graph is $3$-lid-colorable if and only if it is bipartite.
\end{theorem}

\begin{theorem}[\cite{esperet2010locally}]\label{th-bipartite}
Let $G$ and $H$ be two connected bipartite graphs.  Then we have $\chi_{lid}(G\square H) = 3$.
\end{theorem}

\begin{lemma}[\cite{esperet2010locally}]\label{lem-obs}
 A connected graph $G$ is 2-lid-colorable if and only if $G$ has at most two vertices.
\end{lemma}

Lid-coloring is not monotone under
taking subgraphs that is, if $H$ is a subgraph of $G$ then the lid-chromatic number of $H$ may be more than the lid-chromatic number of $G$. 
 
\section{Cartesian product}\label{sec:cartesian}

In this section, we provide an upper bound  on the lid-chromatic number of the Cartesian product of two arbitrary graphs. Next, we determine the lid-chromatic number of $C_m \square P_n$ and $C_m \square C_n$.

\subsection{Cartesian product of two arbitrary graphs}

\begin{lemma}\label{lem-car-nbhd}
 Let $G$ and $H$ be two connected graphs having at least two vertices.   If $(u_1,v_1)$ and $(u_2,v_2)$ are two adjacent vertices in $G \square H$, then we have $N[(u_1,v_1)] \neq N[(u_2,v_2)]$.
\end{lemma}
\begin{proof}
Let $(u_1,v_1)$ and $(u_2,v_2)$ be two adjacent vertices  in $G \square H$. 
Then we have either (a) $u_1=u_2$ and $v_1v_2 \in E(H)$ or (b) $v_1=v_2$ and $u_1u_2 \in E(G)$.

\medskip
\noindent
\textbf{Case~1:} $u_1=u_2$ and $v_1v_2 \in E(H)$. 

As $G$ is connected and $|V(G)| \geq 2$, there exists a vertex $u_3 \in N(u_1)$.
  It is easy to see that $(u_1,v_1)(u_3,v_1)\in E(G\square H)$ and $(u_2,v_2)(u_3,v_1)\notin E(G\square H)$. That is, $(u_3,v_1) \in N[(u_1,v_1)]$ and  $(u_3,v_1) \notin N[(u_2,v_2)]$. Hence, $N[(u_1,v_1)] \neq N[(u_2,v_2)]$.

\medskip
\noindent  
\textbf{Case~2:} $v_1=v_2 $ and $u_1u_2 \in E(G)$.  

 The proof of this case is similar to the proof of Case~1.  
  \qed
\end{proof}

\begin{theorem}\label{thm:generalcart}
Let $G$ and $H$ be two connected graphs having at least two vertices. Then, $\chi_{lid}(G \square H) \leq \chi(G) \chi(H)$.
\end{theorem}
\begin{proof}
Let $\chi(G)=k_1 \geq 2$ and $\chi(H)=k_2\geq 2$.
Let $f_G: V(G) \rightarrow [k_1]$ and $f_H: V(H) \rightarrow [k_2]$ are proper colorings of $G$ and  $H$ respectively. Using the colorings $f_G$ and $f_H$, we construct a lid-coloring of $G \square H$. Define a coloring $g:V(G\square H)\rightarrow [k_1] \times [k_2]$ such that  
for each $(u,v)\in V(G\square H)$,  $g((u,v))=(f_G(u),f_H(v))$.  Now, we show that $g$ is a lid-coloring of $G \square H$.

Let $(u_1,v_1)$ and $(u_2,v_2)$ be two adjacent vertices of $G \square H$. We know that either (a) $u_1=u_2$ and $v_1v_2 \in E(H)$ or (b) $v_1=v_2$ and $u_1u_2 \in E(G)$.

\medskip
\noindent
\textbf{Case~1:} $u_1=u_2$ and $v_1v_2 \in E(H)$.

In this case $g((u_1,v_1))\neq g((u_1,v_2))$ because $f_H(v_1) \neq f_H(v_2)$. 
From Lemma~\ref{lem-car-nbhd}, we know that $N[(u_1,v_1)] \neq N[(u_1,v_2)]$ and $(u_3,v_1) \in N[(u_1,v_1)] \setminus N[(u_1,v_2)]$.  Notice that $g((u_3,v_1))=(f_G(u_3),f_H(v_1))$. It is easy to see that the color $g((u_3,v_1))$ is not assigned to any vertex of $N[(u_1,v_2)]$. That is $g(N[(u_1,v_1)]) \neq g(N[(u_1,v_2)])$. 

\medskip
\noindent
\textbf{Case~2:} $v_1=v_2$ and $u_1u_2 \in E(G)$.

 The proof of this case is similar to the proof of Case~1.  \qed

\end{proof}
The bound presented in the above theorem can be improved by merging two distinct color classes to a single color class. 
\begin{corollary}
Let $G$ and $H$ be two connected graphs having at least two vertices such that $\chi(G)=k_1$ and $\chi(H)=k_2$. Then, $\chi_{lid}(G \square H) \leq k_1k_2  -1$. 
\end{corollary}
\begin{proof}
    Let $g$ be a lid-coloring of $G \square H$ as defined in Theorem \ref{thm:generalcart}. 
    We define a coloring $f:V(G \square H)\rightarrow ([k_1] \times [k_2] )\setminus (k_1, k_2)$ 
    as follows.

 $$
    f((u,v)) = 
        \begin{cases} 
				g((u,v)) & \text{if $g((u,v)) \neq (k_1,k_2)$;} \\
				(1,1) &\text{if $g((u,v))= (k_1,k_2)$.} \\
				               
			\end{cases}
			$$
    
We show that $f$ is a lid-coloring of $G \square H$. 
  Let $e=(u_1,v_1)(u_2,v_2)$ be an arbitrary edge of $G \square H$. That is, either (a) $u_1=u_2$ and $v_1v_2 \in E(H)$ or (b) $v_1=v_2$ and $u_1u_2 \in E(G)$.

\medskip
\noindent
\textbf{Case~1:} $u_1=u_2$ and $v_1v_2 \in E(H)$.

    Let $e=(u_1,v_1)(u_1,v_2)$ be an arbitrary edge of $G \square H$. If $g((u_1,v_1))$ and $g((u_1,v_2))$ are not equal to $(k_1,k_2)$ then clearly  $f((u_1,v_1)) \neq f((u_1,v_2))$.
    Suppose $g((u_1,v_1))=(k_1,k_2)$ and $g((u_1,v_2))=(k_1,p)$, where $p \neq k_2$. Then $f((u_1,v_1))=(1,1)$ and $f((u_1,v_2))=(k_1,p)$. As $k_1 \neq 1$, 
    $f((u_1,v_1)) \neq f((u_1,v_2))$. 
  
From Lemma~\ref{lem-car-nbhd}, we know that $N[(u_1,v_1)] \neq N[(u_1,v_2)]$.
 If $(k_1,k_2) \notin g(N[(u_1,v_1)]\cup N[(u_1,v_2)])$ then clearly we have $f(N[(u_1,v_1)]) \neq f(N[(u_1,v_2)])$. 
 Suppose, $g((u_1,v_1))=(k_1,k_2)$ and $g((u_1,v_2))=(k_1,p)$, where $p \neq k_2$. Then $f((u_1,v_1))=(1,1)$, $f((u_1,v_2))=(k_1,p)$. As $G$ is connected, there exists vertex $u_3 \in V(G)$ such that $u_1u_3 \in E(G)$. Clearly the vertex $(u_3,v_1)$  is adjacent to  $(u_1,v_1)$ and not adjacent to $(u_1, v_2)$, and  
 $f((u_3,v_1))=(q,k_2)$, where $q \neq k_1$. 
 Notice that the color $(q,k_2) \in f(N[(u_1,v_1)]) \setminus f(N[(u_1,v_2)])$
 as $q \neq k_1$ and $p \neq k_2$.

 Similarly, we can show that $f(N[(u_1,v_1)]) \neq f(N[(u_1,v_2)])$ for the case when  $g((u_1,v_1))$ and $g((u_1,v_2))$ not equal to $(k_1,k_2)$ but $(k_1,k_2) \in g(N[(u_1,v_1)]\cup N[(u_1,v_2)])$.

\medskip
\noindent
\textbf{Case~2:} $v_1=v_2$ and $u_1u_2 \in E(G)$.

 The proof of this case is similar to the proof of Case~1.  \qed

\end{proof}

The bound given in the above corollary is sharp when $G=C_3$ and $H=C_4$ as $\chi_{lid}(C_3 \square C_4)=5$ (see Fig~\ref{C3Cn}), $\chi(C_3)=3$ and $\chi(C_4)=2$.

\subsection{Cartesian product of a cycle and a path}
Esperet et al.~\cite{esperet2010locally} showed that for any two bipartite graphs $G$ and $H$ without isolated vertices, $\chi_{lid}(G \square H) = 3$. As a corollary, we can see that the lid-chromatic number of Cartesian product of two paths is three.  

Taking the work forward, we study lid-coloring of Cartesian product of a path and a cycle, and Cartesian product of two cycles. 
  
\begin{theorem}\label{th-CmPn}
For every pair of positive integers $m$ and $n$, where $m \geq 3$, $n \geq 2$, we have
\[
    \chi_{lid}(C_m \square P_n)= 
\begin{cases}
    5& \text{if $m=3$ and $n \geq 2$;}\\
    4& \text{if $m$ is odd, $m \geq 5$ and $n \geq 2$;}\\
    3& \text{if $m$ is even and $n \geq 2$.}\\
\end{cases}
\]
    
\end{theorem}

\begin{proof}
We divide the proof into three cases as described below.

\medskip
\noindent
\textbf{Case~1:} When $m=3$ and $n \geq 2$.

Let $G=C_3 \square P_n$, $V(C_3) = \{u_1, u_2,u_3\}$,  $V(P_n) = \{v_1,v_2,\ldots,v_n\}$   and $V(G)=\{(u_1,v_i), (u_2, v_i), (u_3, v_i)~|~ i \in [n]\}$.
A $5$-lid-coloring of $C_3 \square P_n$ is illustrated in Fig~\ref{fig:C3Pn}.
Thus 
$\chi_{lid}(C_3 \square P_n) \leq 5$.

Next, we show that $\chi_{lid}(G) \geq 5$. Let $X=\{(u_1,v_1), (u_2, v_1), (u_3, v_1)\}$. Clearly the graph $G[X]$  induced by vertices of $X$,   
is isomorphic to $C_3$, and hence $\chi_{lid}(G) \geq  3$. From 
Lemma \ref{lem-car-nbhd}, 
every pair of vertices $u,v \in X$ have distinct closed neighborhoods.  
Hence, to maintain distinct set of colors in $N[u]$ and $N[v]$  
at least two new colors 
must be assigned to the vertices of 
$\{(u_1,v_2), (u_2, v_2), (u_3, v_2)\}$. 
Therefore, any lid-coloring of $G$ uses at least five colors. Thus $\chi_{lid}(G) = 5$. 

 \begin{figure}[t] 
 \captionsetup[subfigure]{justification=centering}
	\begin{subfigure}[t]{0.5\textwidth}	
		
		$$
			\includegraphics[trim=0cm 19.5cm 16cm 1.8cm, clip=true, scale=1]{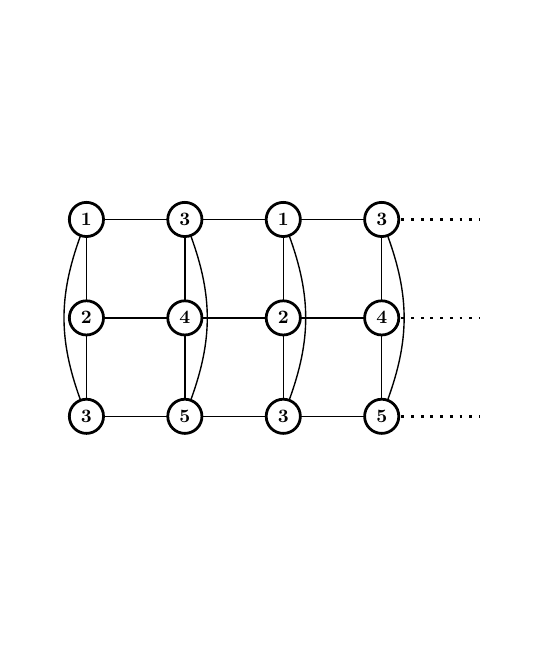}
		$$
	\caption{$C_{3} \square P_n$}
	\label{fig:C3Pn}
	\end{subfigure}
	\begin{subfigure}[t]{0.5\textwidth}
		
		$$
			\includegraphics[trim=7cm 17.5cm 1cm 3.8cm, clip=true, scale=1]{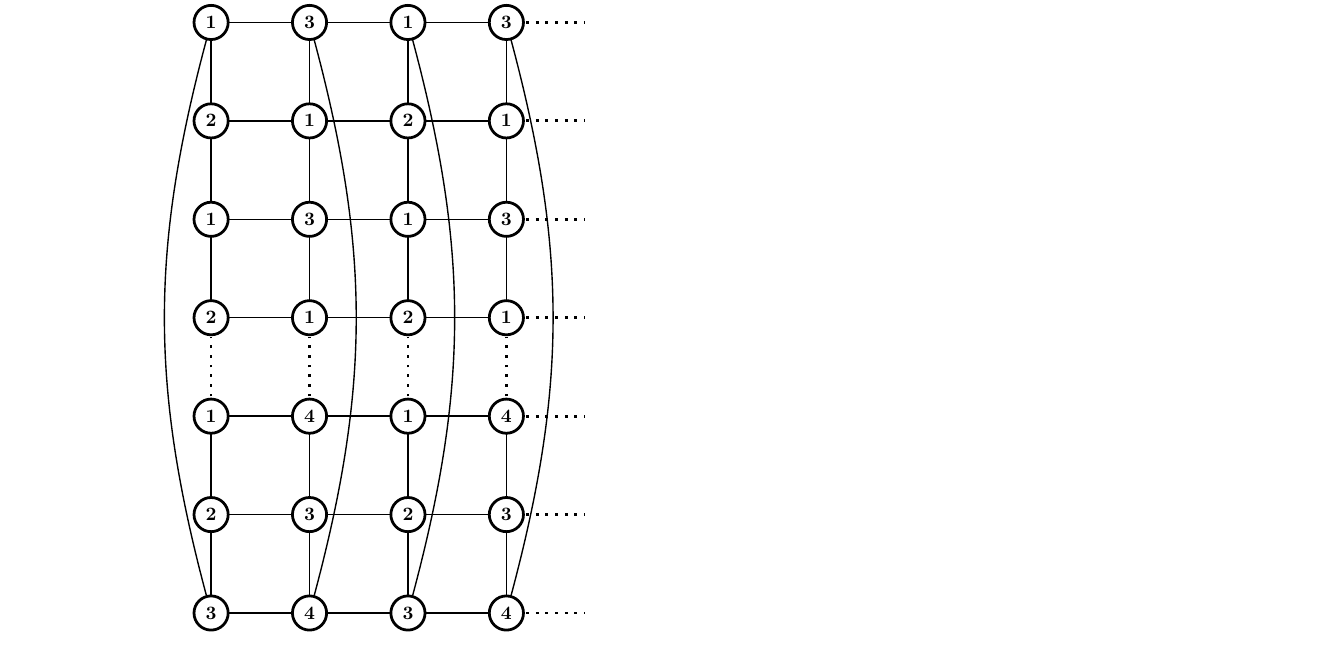}
		$$
	\caption{$C_{m} \square P_n$}
		\label{fig:CmPn}
	\end{subfigure}

	\caption{~ (a) A $5$-lid coloring of $C_3 \square P_n$ for $n \geq 2$,  and (b) A $4$-lid coloring of $C_m \square P_n$, when $m$ is odd, $m \geq 5$ and $n \geq 2$.}
\end{figure}

\medskip
\noindent
\textbf{Case~2:} When $m \geq 5$  is odd and $n\geq 2$.

A $4$-lid coloring of $C_m \square P_n$ is illustrated in  Fig~\ref{fig:CmPn}. Hence, $\chi_{lid}(C_m \square P_n) \leq 4$. 
Suppose $\chi_{lid}(C_m \square P_n)  \leq 3$. 
Then from Lemma~\ref{lem-bipartite-triangle}, 
 $C_m \square P_n$ should be either a triangle or a bipartite graph, which is a contradiction. 
 Hence, $\chi_{lid}(C_m \square P_n) = 4 $.

\medskip
\noindent
\textbf{Case~3:} When $m$ is even and  $n \geq 2$. 

Since $C_m$ and $P_n$ are bipartite, 
from Theorem~\ref{th-bipartite}, 
we get $\chi_{lid}(C_m \square P_n)=3$. \qed

\end{proof}

\subsection{Cartesian product of two cycles}

In this subsection, we study lid-coloring of the Cartesian product of two cycles.

\begin{lemma}~\label{lem-C3_Cn}
 For every positive integer  $n \geq  3$, we have $\chi_{lid}(C_3 \square C_n)=5$.
\end{lemma}
\begin{proof}
 
A $5$-lid-coloring of $C_3 \square C_n$ is illustrated in Fig~\ref{C3Cn}. 
 By following the lines of Case~1 of Theorem~\ref{th-CmPn}, we can show that $\chi_{lid}(C_3 \square C_n) \geq 5$. Hence, we have  $\chi_{lid}(C_3 \square C_n) = 5$.     \qed

\end{proof}


\begin{figure}[t] 
 \captionsetup[subfigure]{justification=centering}
	\begin{subfigure}[t]{0.4\textwidth}
		$$
			\includegraphics[trim=8cm 20.5cm 1.5cm 2.5cm, clip=true, scale=0.9]{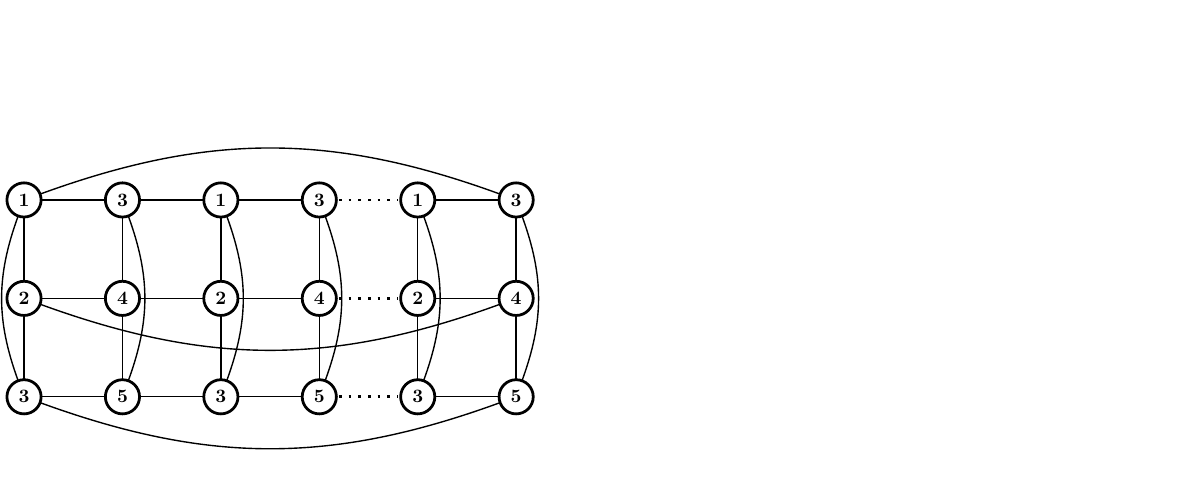}
		$$
	\caption{
 $C_{3} \square C_n$, 
 $n$ is even}	\label{C3Cn(even)}
	\end{subfigure}
	\begin{subfigure}[t]{0.4\textwidth}
	
		$$
			\includegraphics[trim=6.7cm 20.5cm 2cm 2.5cm, clip=true, scale=0.9]{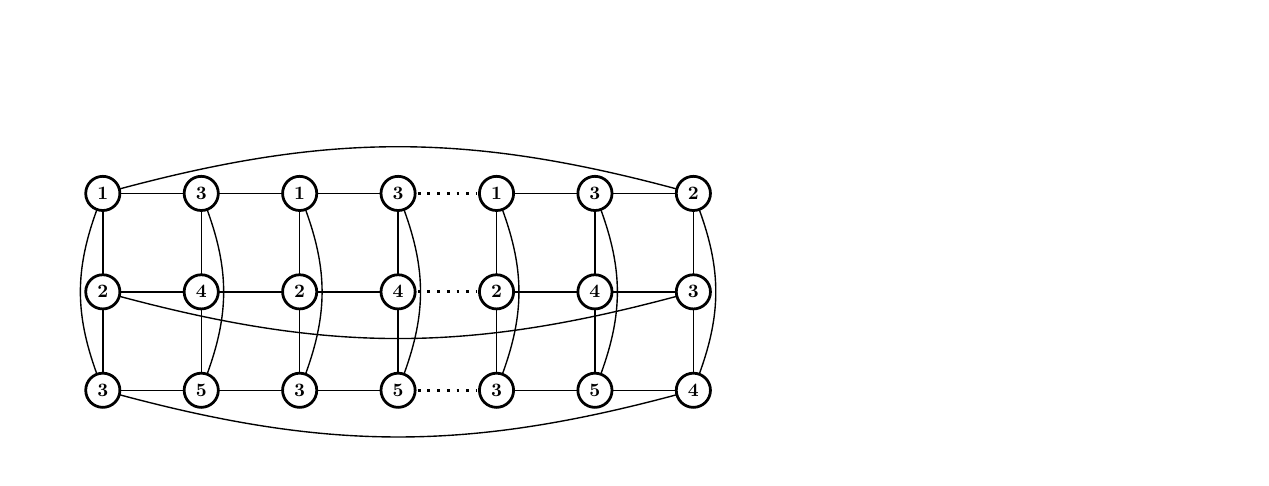}
		$$
	\caption{
 $C_{3} \square C_n$, 
 $n$ is odd}
	\label{C3Cn(odd)}
	\end{subfigure}

	\caption{~(a) A $5$-lid-coloring of $C_3 \square C_n$, when $n$ is even, and (b) A $5$-lid-coloring of $C_3 \square C_n$, when $n$ is odd.}
	\label{C3Cn}
\end{figure}


\begin{lemma}~\label{lem-Ceven_Ceven}
 For every pair of even positive integers $m$ and $n$ such that $3 \leq m \leq n$, we have  $\chi_{lid}(C_m \square C_n)=3$.
\end{lemma}
\begin{proof}
 
The proof follows from Theorem~\ref{th-bipartite} as both $C_m$ and $C_n$ are bipartite. \qed

\end{proof}


\begin{lemma}\label{lem-atleast3}
 If at least one of $m$ and $n$ is odd, then 
 $\chi_{lid}(C_m \square C_n) \geq 4$.  
\end{lemma}
\begin{proof}
Suppose that $\chi_{lid}(C_m \square C_n) \leq 3$. Then 
from Lemma~\ref{lem-bipartite-triangle}, 
$C_m \square C_n$ is either a triangle or a bipartite graph, which is a contradiction
to the fact that $C_m \square C_n$ is neither a triangle nor bipartite. 
Thus, $\chi_{lid}(C_m \square C_n) \geq 4 $.  
\qed
\end{proof}


\begin{figure}[t] 
 \captionsetup[subfigure]{justification=centering}
	\begin{subfigure}[t]{0.4\textwidth}
		$$
			\includegraphics[trim=5cm 17cm 1cm 3.8cm, clip=true, scale=1]{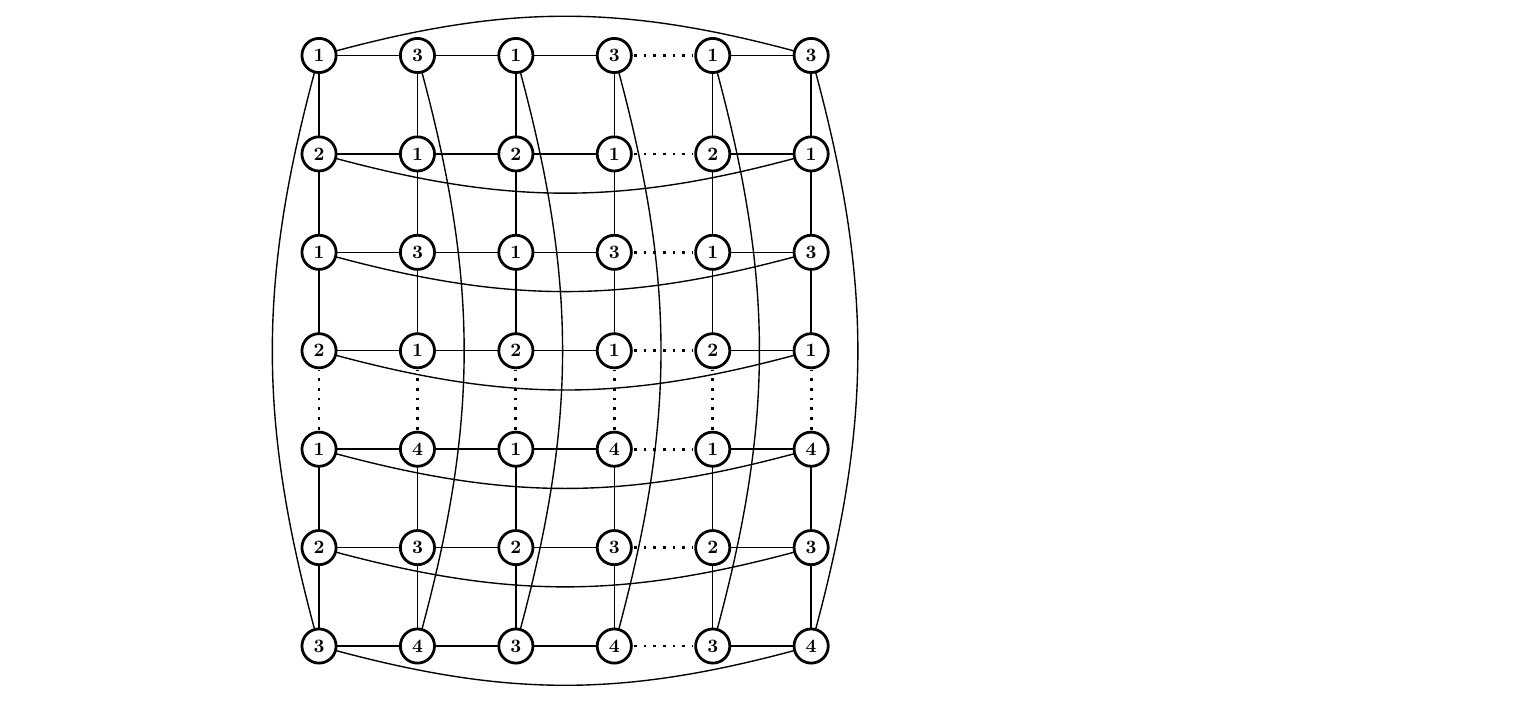}
		$$

	\end{subfigure}

	\caption{~A $4$-lid-coloring of $C_m \square C_n$, where $m (\geq 5)$ is odd and $n$ is even. }
	\label{fig:Codd_Ceven}
\end{figure}


\begin{lemma}~\label{lem-Ceven_Codd}
 Let $m\geq 5$ be an odd integer and $n\geq 4$ be an even integer.  
 Then $\chi_{lid}(C_m \square C_n)=4$.
\end{lemma}
\begin{proof}

From Lemma~\ref{lem-atleast3} we know that $\chi_{lid}(C_m \square C_n) \geq 4 $. A $4$-lid-coloring of $C_m \square C_n$ is shown in Fig \ref{fig:Codd_Ceven}.  Therefore, we get $\chi_{lid}(C_m \square C_n) = 4$.
\qed
\end{proof}

In the rest of this section,  we show that  $\chi_{lid}(C_m \square C_n) = 4$ when both $m$ and $n$ are odd positive integers greater than or equal to five. 
The following result of Sylvester plays a main role in our proofs. 
\begin{lemma}[\cite{sylvester1882subvariants}]\label{lem-main3}
Let $m$ and $n$ be two positive integers that are relatively prime. Then for every integer $k\geq(n-1)(m-1)$, there exist non-negative integers $\alpha$ and $\beta$ such that $k = \alpha n + \beta m$.
\end{lemma}

\begin{lemma}\label{lem:mn12}
 For every pair of odd positive integers $m$ and $n$, where $12 \leq m \leq n$, we have
 $\chi_{lid}(C_m \square C_n)= 4$.
\end{lemma}

\begin{proof}
 From Lemma~\ref{lem-main3}, every positive integer $k \geq 12$ can be expressed as a linear combination of $4$ and $5$. We give $4$-lid-colorings of $C_4 \square C_4$, 
 $C_4 \square C_5$, 
 $C_5 \square C_4$, and 
 $C_5 \square C_5$ in  Fig~\ref{fig-base} such that
 \begin{itemize}
     \item the colors of the first and last columns of $C_4 \square C_4$ and $C_4 \square C_5$ are the same, 
     \item the colors of the first two columns of $C_5 \square C_4$ and 
 $C_5 \square C_5$ are the same, 
 \item the colors of the first two rows of $C_4 \square C_4$ and $C_5 \square C_4$ are the same, and 
 \item the colors of the first two rows of $C_4 \square C_5$ and $C_5 \square C_5$ are the same. 
 \end{itemize}
 
 Therefore by selecting suitable copies of colorings of $C_4 \square C_4$, $C_4 \square C_5$, $C_5 \square C_4$ and $C_5 \square C_5$, 
 we can obtain $4$-lid-coloring 
 of $C_m \square C_n$. From Lemma \ref{lem-atleast3}, we have $\chi_{lid}(C_m \square C_n) \geq 4$. Altogether we have $\chi_{lid}(C_m \square C_n) = 4$.
 For example, a $4$-lid coloring of $C_{13} \square C_{17}$ can be obtained by using suitable copies of colorings of $C_4 \square C_4$, $C_4 \square C_5$, $C_5 \square C_4$ and $C_5 \square C_5$ as shown in Fig~\ref{fig:C13_C17}. \qed 
\end{proof}

\begin{lemma}\label{lem-C5Cn}
    For every odd positive integer $n \geq 5$, we have
    $\chi_{lid}(C_5 \square C_n) = 4$.
       
   \end{lemma}
\begin{proof}
     From Lemma \ref{lem-main3}, we know that every positive integer $k \geq 12$ can be expressed as a linear combination of 4 and 5. As the first two columns of $C_5 \square C_4$ and $C_5 \square C_5$ are identical (see Fig.~\ref{C5_C4}, \ref{C5_C5}), we can use suitable copies of colorings of $C_5 \square C_4$ and $C_5 \square C_5$ to get a 4-lid-coloring of $C_5 \square C_n$ when $n \geq 12$. For $n \in \{7,9,11\}$, we have given 4-lid-colorings of $C_5 \square C_n$ in  Fig.~\ref{fig-C5}. Also from Lemma \ref{lem-atleast3}, we have  $\chi_{lid}(C_5 \square C_n) \geq 4$. Altogether we have $\chi_{lid}(C_5 \square C_n) = 4$.   \qed
\end{proof}
   
\begin{lemma}\label{lem-C7911}
     For every odd positive integers $m$ and $n$, where $m\in \{7,9,11\}$ and $n \geq m$, we have
    $\chi_{lid}(C_m \square C_n) = 4$.
\end{lemma}

\begin{proof}
    The proof of Lemma \ref{lem-C7911} is similar to the proof of Lemma \ref{lem-C5Cn}.   \qed
\end{proof}

\begin{figure}[H] 
 \captionsetup[subfigure]{justification=centering}
	\begin{subfigure}[t]{0.4\textwidth}
	
		$$
			\includegraphics[trim=7.7cm 20cm 1cm 3.9cm, clip=true, scale=0.9]{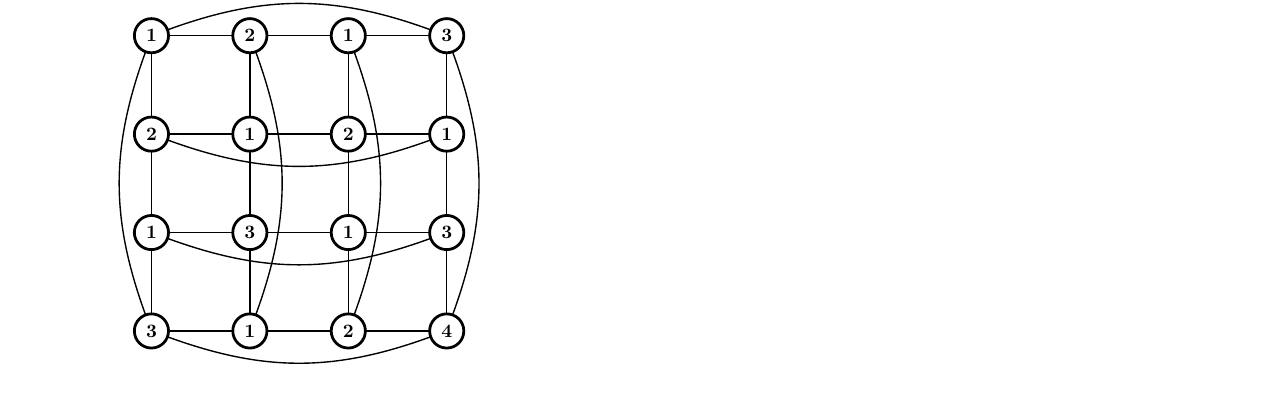}
		$$
	\caption{$C_{4} \square C_4$}
	\label{C4_C4}
	\end{subfigure}
	\begin{subfigure}[t]{0.55\textwidth}
		$$
			\includegraphics[trim=7cm 20cm 1cm 3.9cm, clip=true, scale=0.9]{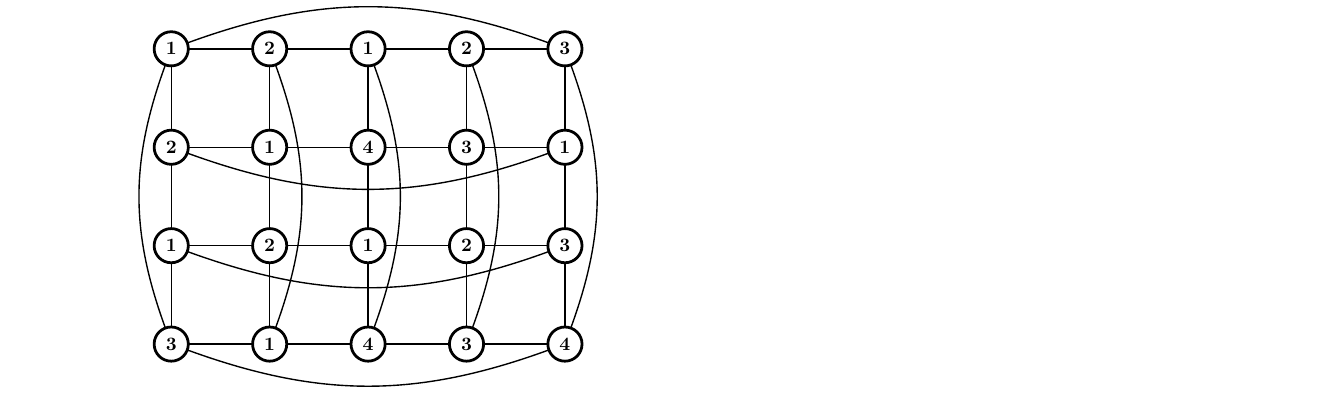}
		$$
	\caption{$C_{4} \square C_5$}
	\label{C4_C5)}
	\end{subfigure}
	
	\vspace{-0.5cm}
	\begin{subfigure}[t]{0.4\textwidth}
	
		$$
			\includegraphics[trim=7.5cm 19cm 1cm 3cm, clip=true, scale=0.9]{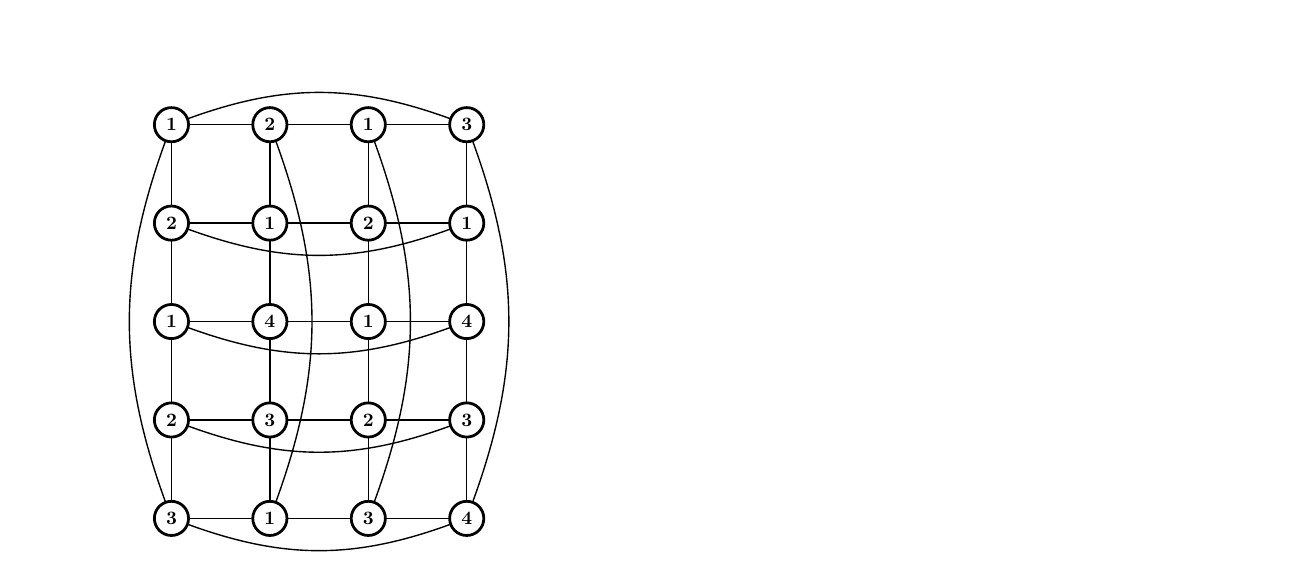}
		$$
	\caption{$C_{5} \square C_4$}
	\label{C5_C4}
	\end{subfigure}
	\begin{subfigure}[t]{0.55\textwidth}
		$$
			\includegraphics[trim=7.2cm 19cm 1cm 3cm, clip=true, scale=0.9]{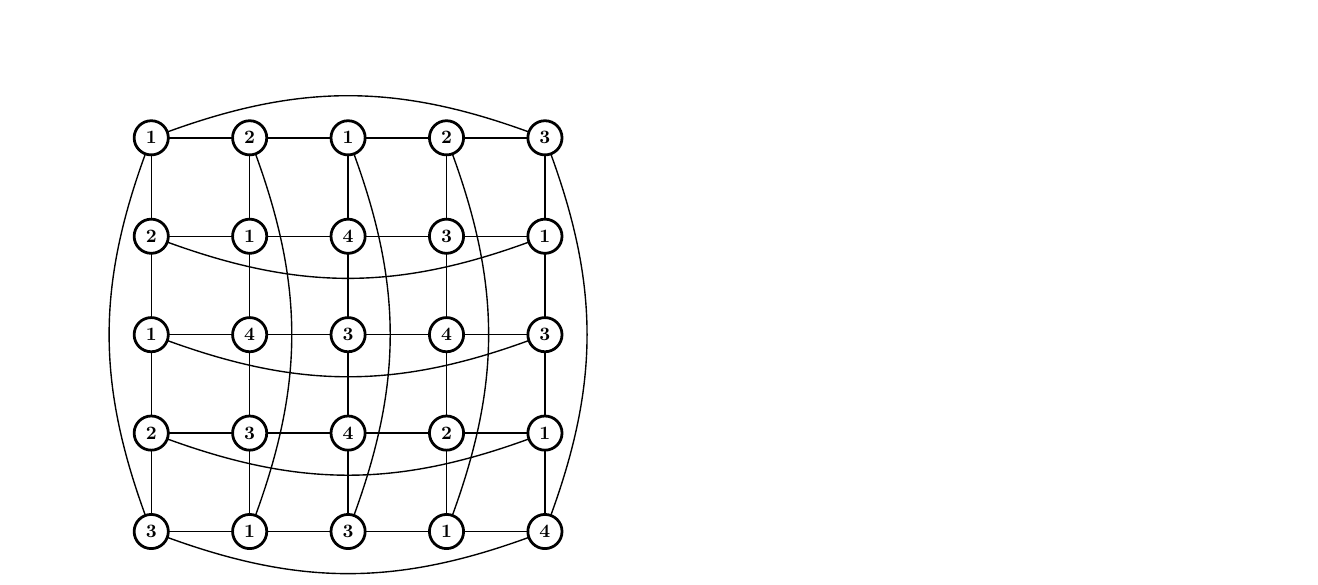}
		$$
	\caption{$C_{5} \square C_5$}
	\label{C5_C5}
	\end{subfigure}

	\caption{~ $4$-lid-colorings of (a)~$C_4 \square C_4$, (b)~$C_4 \square C_5$, (c)~$C_5 \square C_4$ and (d)~$C_5 \square C_5$.}
	\label{fig-base}
\end{figure}

\vspace{-1.5cm}
\begin{figure}[h] 
 \captionsetup[subfigure]{justification=centering}
	\begin{subfigure}[t]{0.4\textwidth}
	
		$$
			\includegraphics[trim=2cm 16.5cm 1cm 4cm, clip=true, scale=0.75]{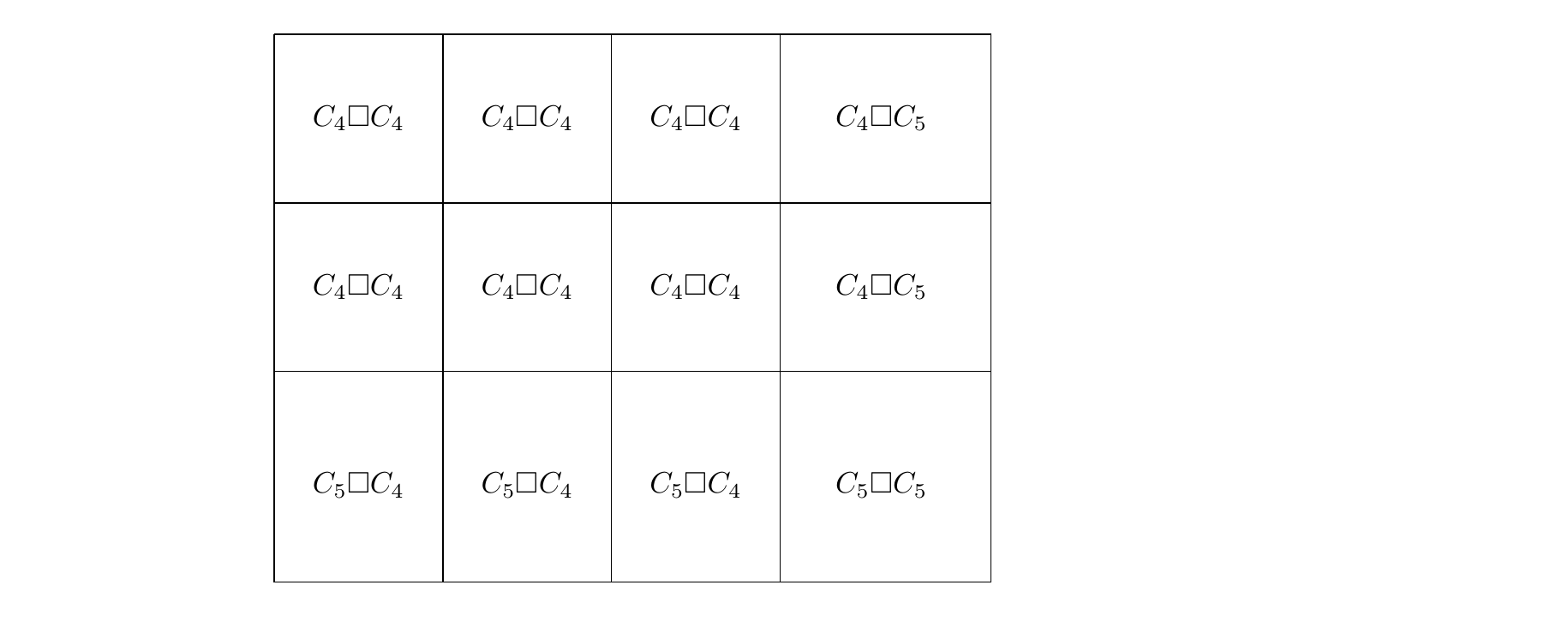}
		$$
	\end{subfigure}

	\caption{~A $4$-lid-coloring of $C_{13} \square C_{17}$ obtained by using suitable copies of colorings $C_4 \square C_4$, $C_4 \square C_5$, $C_5 \square C_4$ and $C_5 \square C_5$.}
	\label{fig:C13_C17}
\end{figure}


\begin{theorem}~\label{th-CmCn}
Let $m$ and $n$ be two positive integers  such that $3 \leq m \leq n$. Then we have
\[
    \chi_{lid}(C_m \square C_n)= 
\begin{cases}
    5& \text{$m=3$ and $n \geq 3$;}\\
    3& \text{$m=2p$ and $n=2q$ for some $p,q \in \mathbb{N}$;}\\
   
    4& \text{otherwise. }\\

\end{cases}
\]

\end{theorem}
\begin{proof}
 The proof of the theorem follows from 
 the Lemmas~\ref{lem-C3_Cn}, 
 \ref{lem-Ceven_Ceven}, 
 \ref{lem-Ceven_Codd}, \ref{lem:mn12}, \ref{lem-C5Cn} and \ref{lem-C7911}.  \qed
\end{proof}





\section{Tensor product}\label{sec:tensor}
In this section, we give an upper bound  on lid-chromatic number of tensor product of two arbitrary graphs. Next, we give lid-chromatic number of $P_m \times P_n$, $C_m \times P_n$ and $C_m \times C_n$.

\subsection{Tensor product of two arbitrary graphs}
Let $G$ and $H$ be two graphs having at least two vertices. If both $G$ and $H$ have exactly two vertices then $G \times H$ contains four vertices and we can find  $\chi_{lid}(G \times H)$ trivially. Therefore, in this section we assume that at least one of $G$ or $H$ contains at least three vertices. 

\begin{lemma}\label{lem-tensor-nbhd}
 Let $G$ and $H$ be two connected graphs such that either $G$ or $H$ has at least three vertices.   If $(u_1,v_1)$ and $(u_2,v_2)$ are two adjacent vertices in $G \times H$, then we have $N[(u_1,v_1)] \neq N[(u_2,v_2)]$.
\end{lemma}

\begin{proof}
Without loss generality, we assume that $H$ has at least three vertices. 
Let $(u_1,v_1)$ and $(u_2,v_2)$ be two adjacent vertices of 
$G \times H$. We know that $u_1u_2 \in E(G)$ and $v_1v_2 \in E(H)$.
As $H$ is connected and $|V(H)| \geq 3$, we have that  degree of either $v_1$ or $v_2$ is at least two. Without loss of generality assume that degree of $v_2$ is at least two and $\{v_1,v_3\} \subseteq N(v_2)$. Then
it is easy to see that $(u_1,v_3)(u_2,v_2)\in E(G\times H)$ and $(u_1,v_3)(u_1,v_1)\notin E(G\times H)$. That is $(u_1,v_3) \in N[(u_2,v_2)]$ and  $(u_1,v_3) \notin N[(u_1,v_1)]$.  \qed
\end{proof}

We call an edge $e=uv$ of $G \times H$ as \emph{bad} with respect to a coloring $g$ if $N[u] \neq N[v]$ but $g(N[u]) = g(N[v])$, otherwise $e$ is called \emph{good}. 

Let $\chi(G)=k_1$ and $\chi(H)=k_2$.
Let $f_G: V(G) \rightarrow [k_1]$ and $f_H: V(H) \rightarrow [k_2]$ are proper colorings of $G$ and  $H$ respectively. Define  a coloring $g:V(G\times H)\rightarrow [k_1]\times [k_2]$ such that 
for each $(u,v)\in V(G \times H)$,  $g((u,v))=(f_G(u),f_H(v))$.

\begin{lemma}\label{lem-tensor-two}
    Let  $e=(u_1,v_1)(u_2,v_2)$ be an edge in $G \times H$ and $g$ be a coloring of $G \times H$ as defined above.  If $e$ is bad with respect to $g$ then
    $g(N[(u_1,v_1)])=g(N[(u_2,v_2)])=\{g((u_1,v_1)),g((u_2,v_2))\}=\{(f_G(u_1),f_H(v_1)),(f_G(u_2),f_H(v_2))\}$.
\end{lemma}

\begin{proof}
    We know from Lemma~\ref{lem-tensor-nbhd} that $N[(u_1,v_1)])\neq N[(u_2,v_2)])$. 
    Since $e$  is bad we have $g(N[(u_1,v_1)])=g(N[(u_2,v_2)])$. 
    Clearly, $\{g((u_1,v_1)),g((u_2,v_2))\} \subseteq g(N[(u_1,v_1)])$ and $\{g((u_1,v_1)),g((u_2,v_2))\} \subseteq g(N[(u_2,v_2)])$. Suppose there exists a vertex $(u,v) \in N[(u_1,v_1)]$ such that $g((u,v))$ is different from both $g((u_1,v_1))$ and $g((u_2,v_2))$.  That is (a) $f_G(u_1) \neq f_G(u)$ and $f_H(v_1) \neq f_H(v)$,  and 
    (b) $f_G(u_2) \neq f_G(u)$ or $f_H(v_2) \neq f_H(v)$.

    It is easy to see that if $(u,v) \in N[(u_1,v_1)]$ then $(u_2,v), (u,v_2) \in N[(u_1,v_1)]$. If  $f_H(v_2) \neq f_H(v)$, then $(f_G(u_2), f_H(v)) \notin g(N[(u_2,v_2)])$ and if $f_G(u_2) \neq f_G(u)$  then $(f_G(u), f_H(v_2)) \notin g(N[(u_2,v_2)])$. In both the cases we get a 
    contradiction to the fact that edge $e$ is  bad  with respect to the coloring $g$. Therefore, we have $g(N[(u_1,v_1)])=g(N[(u_2,v_2)])=\{g((u_1,v_1)),g((u_2,v_2))\}$. 
  \qed  

\end{proof}

\begin{theorem}\label{th-tensor-gen}
 For any two connected graphs $G$ and $H$ such that either $G$ or $H$ has at least three vertices,  $\chi_{lid}(G \times H) \leq \chi(G)  \chi(H)$.   
\end{theorem}
\begin{proof}
Let $\chi(G)=k_1$ and $\chi(H)=k_2$.
Let $f_G: V(G) \rightarrow [k_1]$ and $f_H: V(H) \rightarrow [k_2]$ are proper colorings of $G$ and  $H$ respectively. Using the colorings $f_G$ and $f_H$, we construct a lid-coloring of $G \times H$ in two phases. In the first phase we define a coloring $g:V(G\times H)\rightarrow [k_1]\times [k_2]$ such that  
for each $(u,v)\in V(G \times H)$,  $g((u,v))=(f_G(u),f_H(v))$.

In the second phase we modify the coloring $g$ to get a lid-coloring of $G \times H$. 
The idea behind the second phase coloring is as follows. If an edge $e=(u_1,v_1)(u_2,v_2)$ is bad then from Lemma~\ref{lem-tensor-two} we know that 
$g(N[(u_1,v_1)])=g(N[(u_2,v_2)])=\{g((u_1,v_1)),g((u_2,v_2))\}$. Consider the maximal connected subgraph  $J$ of $G \times H$ induced by the colors $g((u_1,v_1)),g((u_2,v_2))$ containing the vertices $(u_1,v_1)$ and $(u_2,v_2)$. It is easy to see that $J$ is bipartite and we know that every bipartite graph is $4$-lid-colorable. Therefore, we color the subgraph $J$ with four colors $(f_G(u_1)$, $f_H(v_1))$, $(f_G(u_2),f_H(v_2))$, $(f_G(u_1),f_H(v_2))$ and $(f_G(u_2),f_H(v_1))$. The second phase coloring $f$ of $G \times H$ is given in Algorithm~1. Next, we show that $f$ is a lid-coloring of $G \times H$. \vspace{-0.5cm}

\begin{algorithm}[]
\DontPrintSemicolon
\caption{A lid-coloring  
of $G \times H$.}
\label{alg-1}

\KwIn{$G \times H$, $f_G$, $f_H$  and $g$}
\KwOut{ A lid-coloring $f$ of $G \times H$}
 
{

$S=\emptyset$, $Q=V(G \times H)$, $f((u,v))=g((u,v))$ for all $(u,v) \in V(G \times H)$ 

\If {$(G \times H)[Q]$ has a bad edge $e=(u_1,v_1)(u_2,v_2)$ w.r.t. $g$}{
 $f((u_1,v_1))=(f_G(u_1),f_H(v_2))$  \;
 $f((u_2,v_2))=(f_G(u_2),f_H(v_1))$\;
   $S=S \cup (N[(u_1,v_1)] \cup N[(u_2,v_2)])$ \;
   $Q=Q \setminus S$\;
}
}

return (Coloring $f$ of $G \times H$)

\end{algorithm}

\vspace{-0.5cm}

\begin{claim}\label{cla:tensor-proper}
$f$ is a proper-coloring of $G \times H$.
\end{claim}

\begin{proof}
Let $(u_1,v_1)$ and $(u_2,v_2)$ be two adjacent vertices of $G \times H$. We know that $u_1u_2 \in E(G)$ and $f_G(u_1) \neq f_G(u_2)$. 
We have $f((u_1,v_1))=(f_G(u_1),-)$ and $f((u_2,v_2))=(f_G(u_2),-)$. Since $f_G(u_1) \neq f_G(u_2)$, we get $f((u_1,v_1))\neq f((u_2,v_2))$. Therefore  $f$ is a proper coloring of $G\times H$. \qed
\end{proof}

Before proceeding to prove that $f$ is a lid-coloring of $G \times H$, we  classify the edges of $G \times H$ into three categories as follows. An edge $e$ in $G \times H$ is called `fully updated' if  the colors of both its endpoints are changed by Algorithm~1. An edge $e$ is called `partially updated' if the color of only one endpoint of $e$ is changed by Algorithm~1. If both endpoints of $e$ are not changed by Algorithm~1 then we call the edge $e$ a `non-updated' edge.

\begin{claim}\label{cla:tensor-lid}
$f$ is a lid-coloring of $G \times H$.
\end{claim}
\begin{proof}
We show that every edge $e$ of $G \times H$ is good with respect to coloring $f$.

\medskip 
\noindent
\textbf{Case~1:} $e$ is fully updated.

Let $e=(u_1,v_1)(u_2,v_2)$. Without loss of generality, assume that degree of $(u_2,v_2)$ is at least two in $G \times H$. 
As $e$ is a fully updated edge,  $e$ is bad with respect to $g$. That is,  
$g(N[(u_1,v_1)])=g(N[(u_2,v_2)])=\{(f_G(u_1),f_H(v_1)) $, $(f_G(u_2),f_H(v_2))\}$. Algorithm~1 changes colors of $(u_1,v_1)$ and $(u_2,v_2)$ to $(f_G(u_1),f_H(v_2))$ and $(f_G(u_2),f_H(v_1))$ respectively. Also the colors of the vertices in the set $(N[(u_1,v_1)]\cup N[(u_2,v_2)]) \setminus \{(u_1,v_1),(u_2,v_2)\}$ are not changed by Algorithm~1. Therefore,  $(f_G(u_1),f_H(v_1))\notin f(N[(u_1,v_1)])$ as $f_H(v_1) \neq f_H(v_2)$.
However,  $(f_G(u_1),f_H(v_1)) \in N[(u_2,v_2)]$. Therefore, $e$ is good with respect to $f$.

\medskip
\noindent
\textbf{Case~2:} $e$ is partially  updated.

Let $e=(u_2,v_2)(u_3,v_3)$. 
Without loss of generality, assume that the color of $(u_2,v_2)$ is updated by Algorithm~1.
Then there exists an edge  $e'=(u_1,v_1)(u_2,v_2)$ which is fully updated. From Lemma~\ref{lem-tensor-two} we know that $g((u_1,v_1))=g((u_3,v_3))= (f_G(u_1),f_H(v_1))= (f_G(u_3),f_H(v_3))$.

Notice that $(f_G(u_1),f_H(v_2))\in f(N[(u_2,v_2)])$. 
However, the color $(f_G(u_1),f_H(v_2))\notin N[(u_3,v_3)]$ as $f_G(u_1) = f_G(u_3)$ and $f_H(v_2) \neq f_G(v_3)$. Therefore $e$ is good with respect to $f$.

\medskip
\noindent
\textbf{Case~3:} $e$ is  non-updated.

Let $e=(u_3,v_3)(u_4,v_4)$. If Algorithm~1 doesn't update any vertex from the set $N[(u_3,v_3)] \cup N[(u_4,v_4)]$ then clearly $e$ is good with respect to $f$.

Suppose, the color of a vertex $(u_2,v_2) \in N((u_3,v_3))$ is updated by Algorithm~1. 
Then there exists an edge $e'=(u_1,v_1)(u_2,v_2)$ which is fully updated. From Lemma~\ref{lem-tensor-two} we know that $g((u_1,v_1))=g((u_3,v_3))= (f_G(u_1),f_H(v_1))= (f_G(u_3),f_H(v_3))$.

Suppose that $e$ is bad with respect to $f$.  Then 
$f((u_2,v_2))=f((u_4,v_4))=(f_G(u_2),f_H(v_1))$.
That is we have   $f((u_3,v_3))=(f_G(u_1),f_H(v_1))$ and $f((u_4,v_4)) =(f_G(u_2),f_H(v_1))$, which is a contradiction as $f_H(v_3)=f_H(v_4)$ and $v_3v_4 \in E(H)$. Therefore $e$ is good with respect to $f$.   \qed
\end{proof} 
\qed
\end{proof}

We can easily see that the bound given in the Theorem~\ref{th-tensor-gen} is sharp for $G=H=P_4$.

\subsection{Tensor product for two paths}
We use the following known results on tensor product in our proofs. 

\begin{lemma}[\cite{hammack2011handbook}]\label{lem-tensor0}
    Let $G$ and $H$ be two graphs. If $G$ or $H$ is bipartite then $G \times H$ is bipartite. 
\end{lemma}

\begin{lemma}[\cite{weichsel1962kronecker}]\label{lem-tens1}
    For two connected graphs $G$ and $H$, the tensor product $G\times H$ is connected if and only if either $G$ or $H$ is non-bipartite. 
\end{lemma}

\begin{lemma}[\cite{weichsel1962kronecker}]\label{lem-tens2}
     If $G$ and $H$ are connected bipartite graphs then $G\times H$ has exactly two components.
\end{lemma}

\begin{theorem}
    For every pair of positive integers $m$ and $n$, where $2 \leq m \leq n$, we have

\[
    \chi_{lid}(P_m \times P_n)= 
\begin{cases}
    2& \text{if $m=2$ and $n=2$;}\\
    4& \text{if $m,n \geq 4$  
    are even;}\\
    3& \text{otherwise}  
\end{cases}
\]
\end{theorem}

\begin{proof}
Let $ V(P_m) = \{u_1,u_2,\ldots,u_m\}$, $V(P_n) = \{v_1,v_2,\ldots,v_n\}$ and $V(P_m \times P_n) = \{(u_i,v_j) 
 \mid i \in [m],j\in [n]\}$. 
 
\medskip
\noindent
\textbf{Case~1:} When $m=2$ and $n=2$.

The graph $P_2 \times P_2$ is a disjoint union of  two $P_2$'s. Hence, $\chi_{lid}(P_2 \times P_2) = 2$. 

\medskip
\noindent
\textbf{Case~2:}  When $m,n \geq 4$ are even. 

Using Lemma~\ref{lem-tens1} and Lemma~\ref{lem-tens2} we can see that the graph $P_m \times P_n$ is a disconnected graph having exactly two connected components.  
Let the two connected components be $B_1$ and $B_2$, where $V(B_1)=\{(u_i,v_j)~|~ i+j \text{~is even} \}$ and $V(B_2)=\{(u_i,v_j)~|~ i+j \text{~is odd}\}$. 
As  $m$ and $n$ are even, both $B_1$ and $B_2$ contain exactly two vertices of degree one.  
The two degree one vertices in $B_1$ are $(u_1,v_1)$ and $(u_m,v_n)$.

Suppose, $\chi_{lid}(B_1) =3$ and let $f$ be a $3$-lid-coloring of $B_1$.
It is easy to see that the distance between $(u_1,v_1)$ and $(u_m,v_n)$ is $2q+1$ for some $q \in \mathbb{N}$.  We know $deg((u_1,v_1)) = deg((u_m, v_n)) = 1$. Thus, we have $|f(N[(u_1,v_1)])| = 2$. This implies that $|f(N[(u_2,v_2)])| = 3$, otherwise $f(N[(u_1,v_1)])=f(N[(u_2,v_2)])$, contradicting the fact that $f$ is a lid-coloring. 
Since $|f(N[(u_2,v_2)]|=3$, and $f$ is a $3$-lid-coloring of $B_1$ we get  $|f(N[(u,v)]|=2$ for every $(u,v) \in N((u_2,v_2))$. 
Continuing this way, for all the vertices on any shortest path from $(u_1,v_1)$ to $(u_m, v_n)$, we get $|f(N[(u_m,v_n)])| = 3$, 
which is not possible as $deg((u_m, v_n)) = 1$. This contradicts the assumption that $f$ is a $3$-lid-coloring of $B_1$.

Thus $\chi_{lid}(P_m \times P_n) \geq \chi_{lid}(B_1) \geq 4$. As $P_m \times P_n$ is a bipartite graph, from Theorem~\ref{lem-bipartite} we have $\chi_{lid}(P_m \times P_n) \leq 4$. Therefore, we have $\chi_{lid}(P_m \times P_n) = 4$. 

\medskip
\noindent
\textbf{Case~3:} When $m$ is odd and $n \geq 2$.

A $3$-lid-coloring of $P_m \times P_n$ is given in Fig. \ref{PmPn}. Therefore, we have $\chi_{lid}(P_m \times P_n) \leq 3$. From Lemma \ref{lem-obs}, we know that $\chi_{lid}(P_m \times P_n) \geq 3$.  Altogether, we have $\chi_{lid}(P_m \times P_n) = 3$.

\medskip
\noindent
\textbf{Case~4:} When $m \geq 2$ and $n$ is odd.

As tensor product is commutative, this case is same as Case~3. \qed

\begin{figure}[ht] 
 \centering
		
		\includegraphics[trim=3.3cm 16.5cm 6cm 3cm, clip=true, scale=0.9]{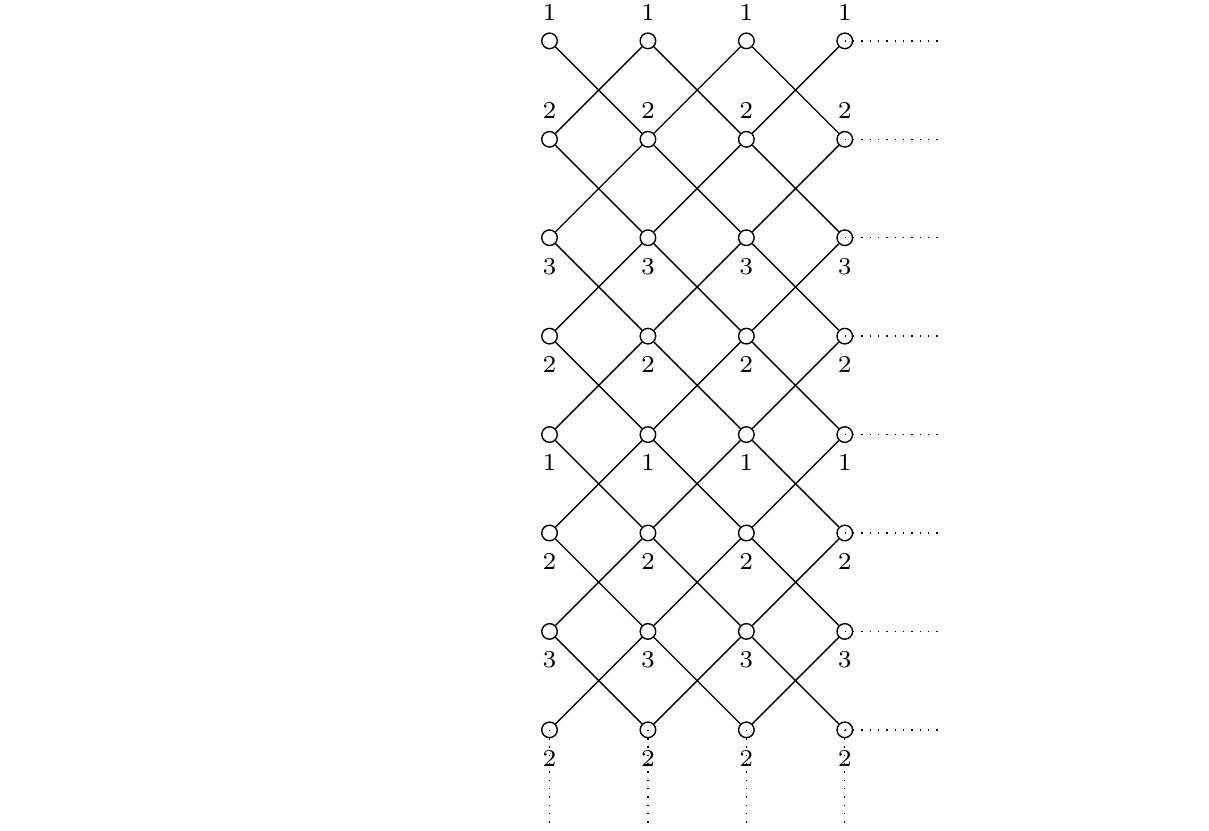}
   
\caption{~A 3-lid-coloring of $P_m \times P_n$.}
 \label{PmPn}
\end{figure}

\end{proof}

\subsection{Tensor product of a cycle and a path}
\begin{theorem} \label{thm:tensor}
Let $m$ and $n$ be two positive integers  such that $m \geq 3$ and $n\geq 2$. Then we have

\[
    \chi_{lid}(C_m \times P_n)= 
\begin{cases}
    3& \text{if $m \geq 3$ and $n$ is odd;}\\
   3 & \text{if $m
   $ is a multiple of $4$ and $n$ is even;}\\
    4& \text{otherwise}
\end{cases}
\]
\end{theorem}

\begin{proof}
    Let $ V(C_m) = \{u_1,u_2,\ldots,u_m\}$, 
    $V(P_n) = \{v_1,v_2,\ldots,v_n\}$ and $V(C_m \times P_n) = \{(u_i,v_j) \mid i \in [m],j\in [n]\}$.\\
    
\noindent
\textbf{Case~1:}
 When $m \geq 3$ and $n$ is odd.

 A 3-lid-coloring of $C_m \times P_n$ is given in Fig. \ref{CmPn}. Therefore, 
 $\chi_{lid}(C_m \times P_n) \leq 3$. 
 From Lemma~\ref{lem-obs}, we know that $\chi_{lid}(C_m \times P_n) \geq 3$. 
 Thus, 
 $\chi_{lid}(C_m \times P_n) = 3$. 

\medskip
\noindent
\textbf{Case~2:} When $m$ is a multiple of 4 and 
$n$ is even.

When $n=2$, the graph 
$C_m \times P_n$ is disconnected in which each connected component is a copy of $C_m$.  
Therefore, 
 $\chi_{lid}(C_m \times P_n)= \chi_{lid}(C_m)$.  
 From Lemma \ref{lem-cycle},  
 we have $\chi_{lid}(C_m \times P_n) = 3$.

When  $n \geq 4$, a 3-lid-coloring of $C_m \times P_n$ is given in Fig. \ref{CmPn2}. Therefore, 
$\chi_{lid}(C_m \times P_n) \leq 3$. From Lemma~\ref{lem-obs}, we have $\chi_{lid}(C_m \times P_n) \geq 3$. Thus 
$\chi_{lid}(C_m \times P_n) = 3$.

\medskip
\noindent
\textbf{Case~3(a):} When $m$ 
is not a multiple of 4, and both $m$ and $n$ are even.

When $n=2$, from Lemma \ref{lem-cycle} we get  $\chi_{lid}(C_m \times P_n)=\chi_{lid}
(C_m)=4$. The arguments are similar to the above case when $n=2$. 

Now, we  deal with the case when $n \geq 4$. From Lemma~\ref{lem-tens1} and Lemma~\ref{lem-tens2} we get that the
 graph $C_m \times P_n$ is a disconnected bipartite graph and contains exactly two connected components. 
Let the two connected components be $B_1$ and $B_2$, where $V(B_1)=\{(u_i,v_j)~|~ i+j \text{~is even} \}$ and $V(B_2)=\{(u_i,v_j)~|~ i+j \text{~is odd}\}$. 

Suppose that $\chi_{lid}(B_1)=3$ and let $f$ be a $3$-lid-coloring of $B_1$. 
Consider a vertex 
$(u_1,v_1)$. We divide the proof into two cases based on the number of colors used by $f$ in the closed neighborhood of $(u_1,v_1)$.

\begin{figure}[h] 
 \captionsetup[subfigure]{justification=centering}
    \includegraphics[trim=4.5cm 21cm 7cm 3cm, clip=true, scale=1]{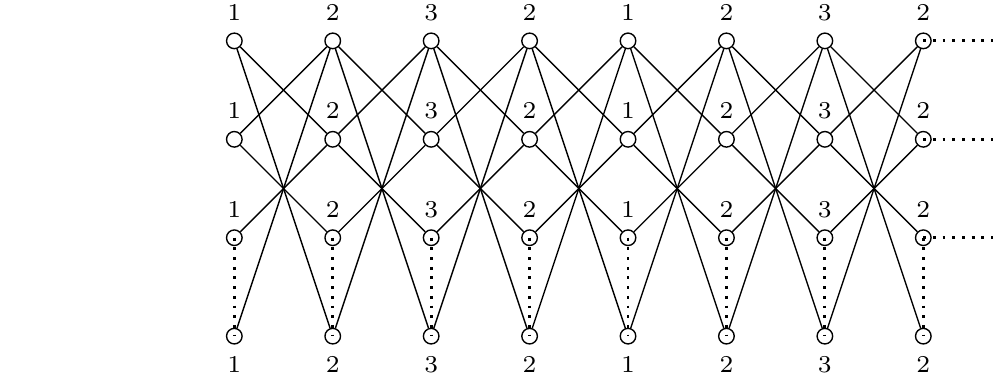}
    \caption{~ A $3$-lid-coloring of $C_m \times P_n$, when $n$ is odd, is obtained from the figure by selecting first $m$ rows and $n$ columns following the above pattern. }
    \label{CmPn}
    \end{figure}
		\begin{figure}
		\includegraphics[trim=3cm 15.5cm 6cm 2.5cm, clip=true, scale=0.8]{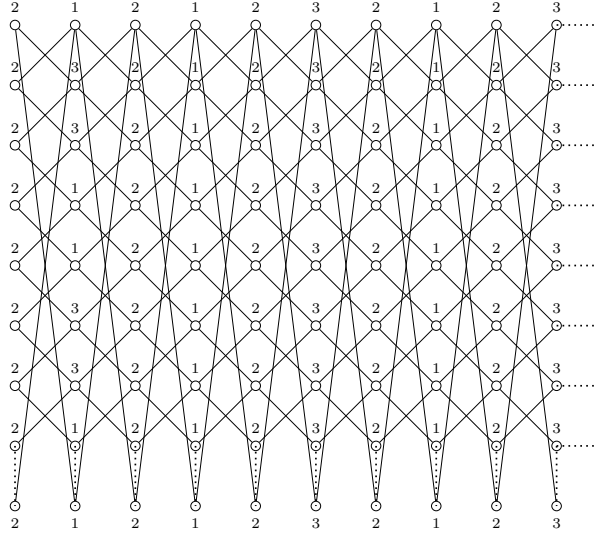}
  \caption{~ A 3-lid-coloring of $C_m \times P_n$, when $m$ is a multiple of 4
and $n\geq 4$ is even, is obtained 
from the figure by selecting first $m$ rows and $n$ columns following the above pattern. }
    \label{CmPn2}
	\end{figure}

\medskip
\noindent
\textbf{Case~(A):} $|f(N[(u_1,v_1)])| = 2$.

We know that $N((u_1,v_1))=\{(u_2,v_2), (u_m,v_2)\}$.
As $f$ is a lid-coloring, we have $|f(N[(u_2,v_2)])|=|f(N[(u_m,v_2)])| =3$. Next, we know that $N((u_2,v_2))=\{(u_1,v_1), (u_3,v_1),(u_1,v_3),(u_3,v_3)\}$. Since 
$|f(N[(u_2,v_2)])|=3$, and $f$ is a $3$-lid-coloring we have $|f(N[(u,v)
])|=2$ for every $(u,v) \in N((u_2,v_2))$.
Continuing the arguments this way, we get 
$|f(N[(u_i,v_j)])| = 2$ when both $i$ and $j$ are odd and $|f(N[(u_i,v_j)])| = 3$ when both $i$ and $j$ are even. 

Since $n-1$ is odd, we have $|f(N[(u_i,v_{n-1})])| = 2$, for each $i \in \{1,3, \ldots, m-1\}$. 
That is  
all the vertices in the set 
$\{(u_i,v_{n-2}),(u_i,v_n)~|~ i \in \{2,4, \ldots, m\} \}$
are assigned the same color by $f$.
Since $|f(N[(u_i,v_n)])| = 3$, for each $i\in \{2,4 \ldots, m-2\}$ and $N((u_i,v_n))= \{(u_{i-1},v_{n-1}), (u_{i+1},v_{n-1})\}$,   
therefore $f((u_{i-1},v_{n-1})) \neq f((u_{i+1},v_{n-1}))$.

As $f$ is a $3$-lid-coloring of $B_1$, we get that, all the vertices in the set $\{(u_{i},v_{n-1}) ~|~ i \in \{1,5,9 \ldots, m-1\} \}$ are assigned the same color by $f$. Similarly, all the vertices in the set $\{(u_{i},v_{n-1}) ~|~ i \in \{3,7, \ldots, m-3 \} \}$ are assigned the same color by $f$.

Combining all the above, we get  $f((u_1,v_{n-1})) = f((u_{m-1},v_{n-1}))$.  We know that $N((u_m,v_n)) = \{(u_1,v_{n-1}), (u_{m-1},v_{n-1}) \}$, therefore  we get $|f(N[(u_m,v_n)])| = 2$, which contradicts our assumption that $|f(N[(u_i,v_j)]| = 3$ when both $i$ and $j$ are even. Therefore, $f$ is not a 3-lid-coloring of $B_1$. Thus, 
$\chi_{lid}(C_m \times P_n) \geq \chi_{lid}(B_1) \geq 4$. 
As $C_m \times P_n$ is bipartite, from Theorem~\ref{lem-bipartite} we know that  $\chi_{lid}(C_m \times P_n) \leq 4$. Therefore, we have $\chi_{lid}(C_m \times P_n) = 4$.

\medskip
\noindent
\textbf{Case~(B):} $|f(N[(u_1,v_1)])| = 3$.

Following similar lines as proof of the above case, we can show that $\chi_{lid}(C_m \times P_n) = 4$.

\medskip
\noindent
\textbf{Case~3(b):} When $m$ is odd and $n$ is even.

The proof of this case is similar to the proof of  Case~3(a). \qed

\end{proof}

\subsection{Tensor product of two cycles}
\begin{lemma}\label{lem-tensor-even}
Let $m$ and $n$ be two integers such that  $3 \leq m \leq n$. If at least one of $m$ or $n$ is even then  $ \chi_{lid}(C_m \times C_n) =3$.
\end{lemma}

\begin{proof}
    
In this case, at least one of $C_m$ or $C_n$ is bipartite and hence from Lemma~\ref{lem-tensor0} $C_m \times C_n$ is bipartite. 
From Theorem~\ref{lem-regular}, we know that for $k \geq 4$, a $k$-regular graph is $3$-lid-colorable if and only if it is bipartite. Since $C_m \times C_n$ is a $4$-regular bipartite graph, 
we have that $\chi_{lid}(C_m \times C_n) = 3$. \qed
\end{proof}

For the rest of this section, we deal with the case where both $m$ and $n$ are odd. 
Thus from Lemma~\ref{lem-bipartite-triangle}, 
we have that $\chi_{lid}(C_m \times C_n) \geq 4$.

\begin{lemma}\label{lem-tensor-cc-1}
Let $m$ and $n$ be two odd positive integers such that  $m \geq 9$ and $n \geq 3$. Then we have $ \chi_{lid}(C_m \times C_n) =4$.
\end{lemma}
\begin{proof}
    As $m \geq 9$ is an odd integer, from Lemma \ref{lem-cycle} we know that $\chi_{lid}(C_m)=4$. 
    Let $g$ be a $4$-lid-coloring of $C_m$. We define a $4$-lid-coloring $f$ of $C_m \times C_n$ as $f(u,v)=g(u)$ for every $(u,v) \in V(C_m \times C_n)$. It is easy to see that $f$ is a proper coloring of $C_m \times C_n$. 
    
    Consider two adjacent vertices $(u_1,v_1)$ and $(u_2,v_2)$. From Lemma~\ref{lem-tensor-nbhd} we know  that $N[(u_1,v_1)]\neq N[(u_2,v_2)]$.
    We have $f(N[(u_1,v_1)])=g(N[u_1])$ and  $f(N[(u_2,v_2)])=g(N[u_2])$. Since $u_1u_2 \in E(C_m)$ and $N[u_1] \neq N[u_2]$, we have $g(N[u_1]) \neq g(N[u_2])$. Therefore, 
    $f(N[(u_1,v_1)]) \neq f(N[(u_2,v_2)])$. Hence, $f$ is a $4$-lid-coloring of $C_m \times C_n$. \qed
  
\end{proof}

\begin{lemma}\label{lem-tensor-cc-2}
    $\chi_{lid}(C_m \times C_n)=4$ for the pairs $(m,n) \in \{(3,7),(5,5),(5,7),(7,7)\} $.
\end{lemma}
\begin{proof}
    From Lemma~\ref{lem-bipartite-triangle} we know that $\chi_{lid}(C_m \times C_n) \geq 4$. We have given $4$-lid-colorings of $C_m \times C_n$ for $(m,n) \in \{(3,7),(5,5),(5,7),(7,7)\} $ in Fig~\ref{37}, Fig~\ref{55}, Fig~\ref{57} and Fig~\ref{77} respectively. \qed
    \end{proof}

\begin{lemma}\label{lem-tensor-base}
    $\chi_{lid}(C_3 \times C_3)=\chi_{lid}(C_3 \times C_5)=5$.
\end{lemma}
\begin{proof}
    We have given a $5$-lid-coloring of $C_3 \times C_3$ and $C_3 \times C_5$ in Fig~\ref{35}. We found
   that $\chi_{lid}(C_3 \times C_3)=\chi_{lid}(C_3 \times C_5) = 5$ by performing a tedious case by case analysis. \qed
\end{proof}

\begin{theorem}
Let $m$ and $n$ be two positive integers such that $3 \leq m \leq n$. Then we have
\[
    \chi_{lid}(C_m \times C_n)= 
\begin{cases}
    5& \text{if $m=3$ and $n \in \{3,5\}$;}\\
   4 & \text{if $m = 3$ and $n=7$;}\\
    4& \text{if $m \in \{5,7\}$ and $n \in \{5,7\}$;}\\
    4& \text{if $m \geq 9$, $m$ is odd and $n \geq 3$, $n$ is odd;}\\
    3& \text{otherwise}
\end{cases}
\]
\end{theorem}
\begin{proof}
    The proof follows from the results of Lemmas~\ref{lem-tensor-even},~\ref{lem-tensor-cc-1},~\ref{lem-tensor-cc-2},~\ref{lem-tensor-base}.  \qed
\end{proof}

\bibliographystyle{plain}
\bibliography{ref.bib}
\newpage
\section{Appendix} \label{appendix}
\subsection{Figures related to the Cartesian product of two odd cycles}

\begin{figure}[H] 
 \captionsetup[subfigure]{justification=centering}
	\begin{subfigure}[t]{0.4\textwidth}
	
		$$
			\includegraphics[trim=7cm 18cm 1cm 4cm, clip=true, scale=0.8]{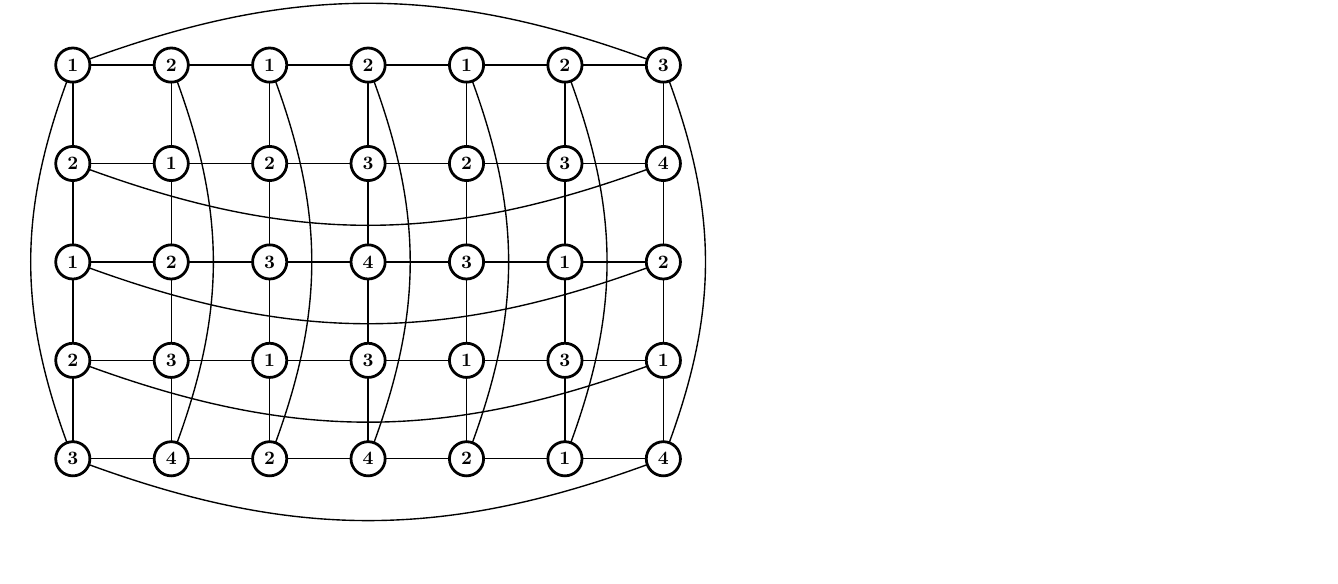}
		$$
	\caption{$C_{5} \square C_7$}
	\label{C4_C4}
	\end{subfigure}
	\begin{subfigure}[t]{0.55\textwidth}
		$$
			\includegraphics[trim=4.5cm 18cm 1cm 4cm, clip=true, scale=0.8]{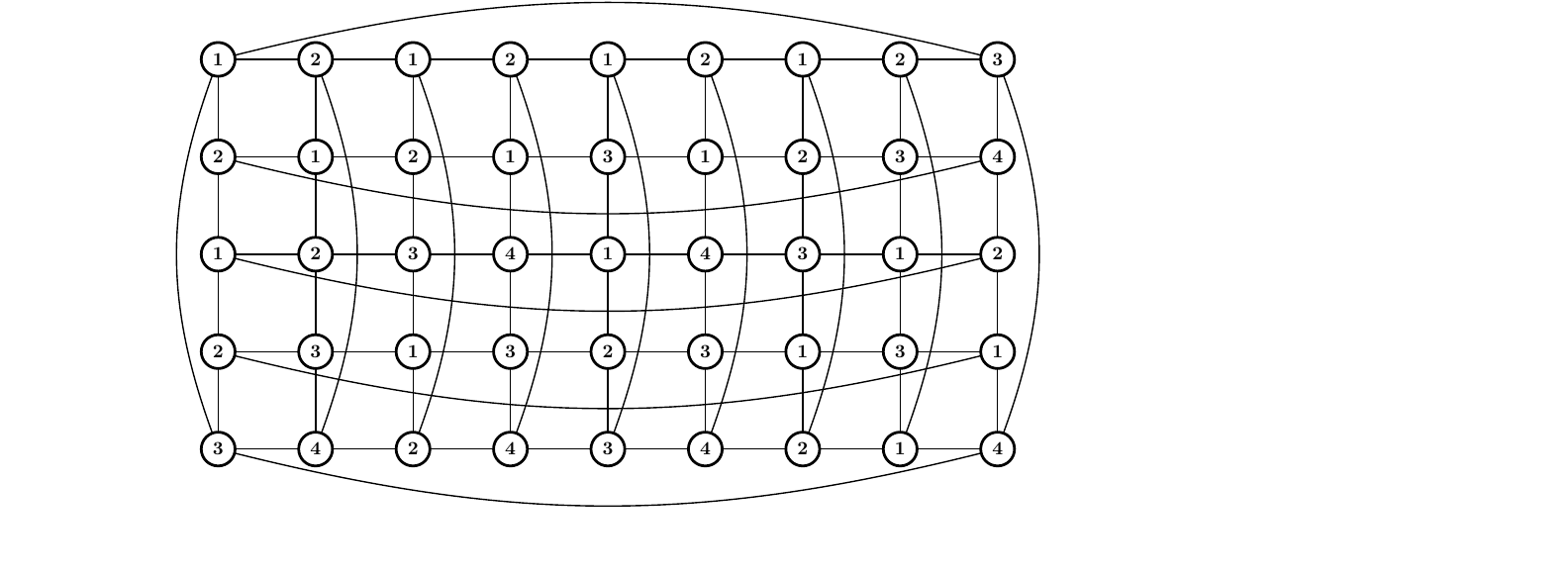}
		$$
	\caption{$C_{5} \square C_9$}
	\label{C4_C5)}
	\end{subfigure}
	\end{figure}
 
 \begin{figure}
	\begin{subfigure}[t]{1\textwidth}
	
		$$
			\includegraphics[trim=3.5cm 18cm 1cm 3cm, clip=true, scale=0.9]{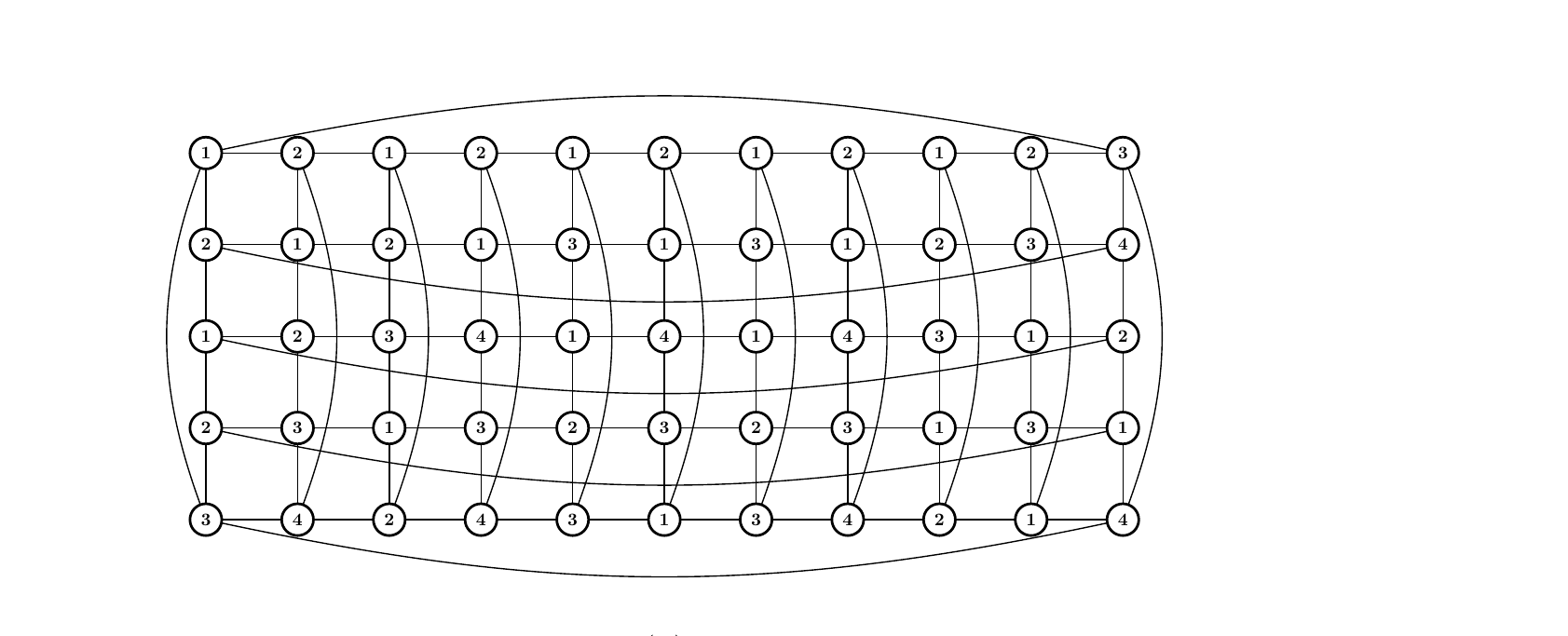}
		$$
	\caption{$C_{5} \square C_{11}$}
	\label{C5_C11}
	\end{subfigure}

	\caption{~ $4$-lid colorings of $C_5 \square C_n$, $n \in \{7,9,11\}$.}
	\label{fig-C5}
\end{figure}


\begin{figure}[H] 
 \captionsetup[subfigure]{justification=centering}

 \vspace{-3cm}

\begin{subfigure}[t]{0.4\textwidth}
		$$
			\includegraphics[trim=8cm 17.3cm 1.5cm 3.9cm, clip=true, scale=0.8]{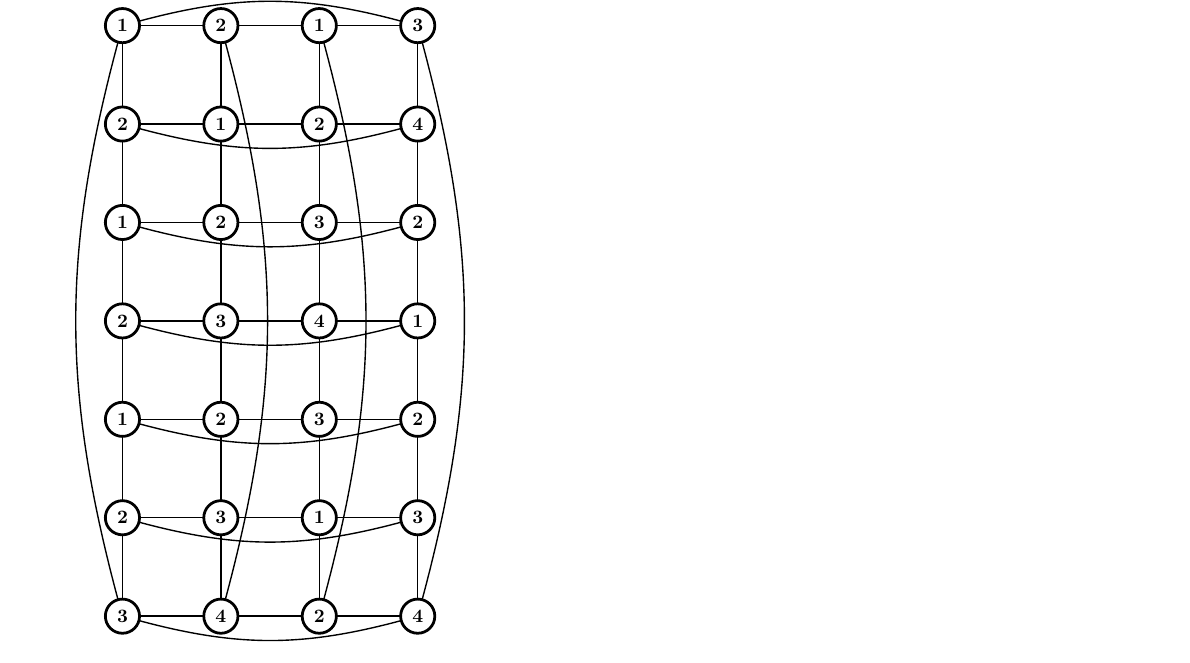}
		$$
	\caption{$C_7 \square C_4$}	
 \label{C7C4}
	\end{subfigure}
	\begin{subfigure}[t]{0.4\textwidth}
	
		$$
			\includegraphics[trim=6.7cm 17.3cm 2cm 3.9cm, clip=true, scale=0.8]{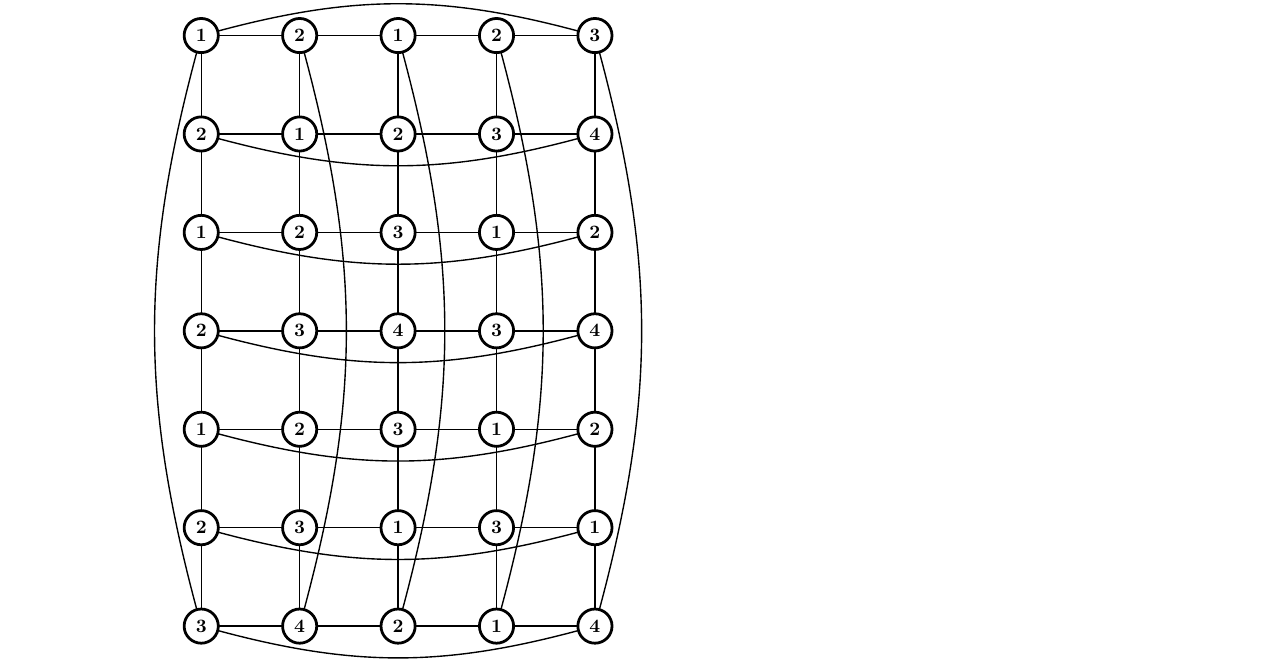}
		$$
	\caption{$C_7 \square C_5$}
	\label{C7C5}
	\end{subfigure}

\vspace{-0.5cm}

	\begin{subfigure}[t]{0.4\textwidth}
	
		$$
			\includegraphics[trim=7.2cm 16.5cm 1cm 2.5cm, clip=true, scale=0.7]{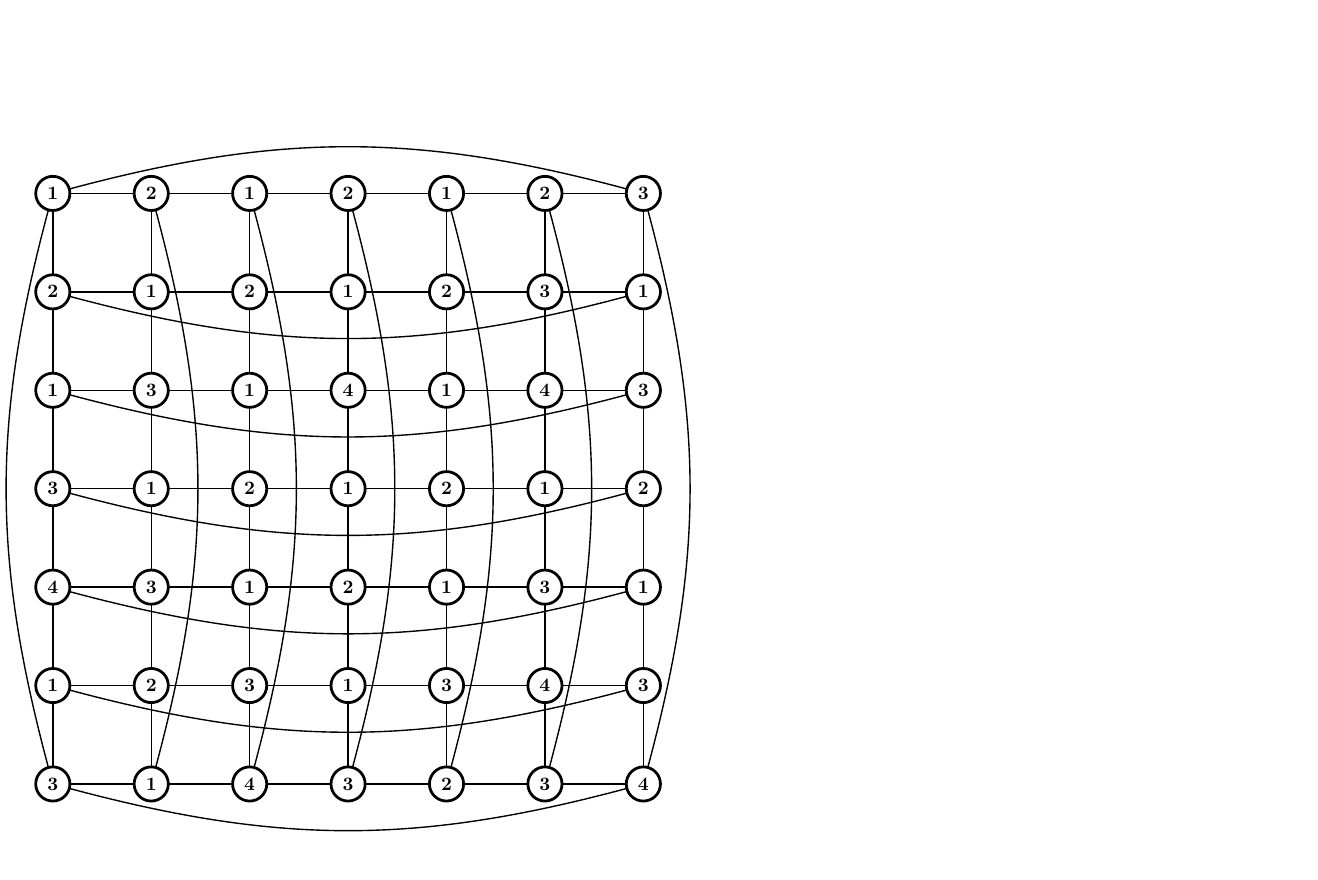}
		$$
	\caption{$C_{7} \square C_7$}
	\label{C7_C7}
	\end{subfigure}
	\begin{subfigure}[t]{0.8\textwidth}
		$$
			\includegraphics[trim=4.5cm 16.5cm 1cm 2.5cm, clip=true, scale=0.7]{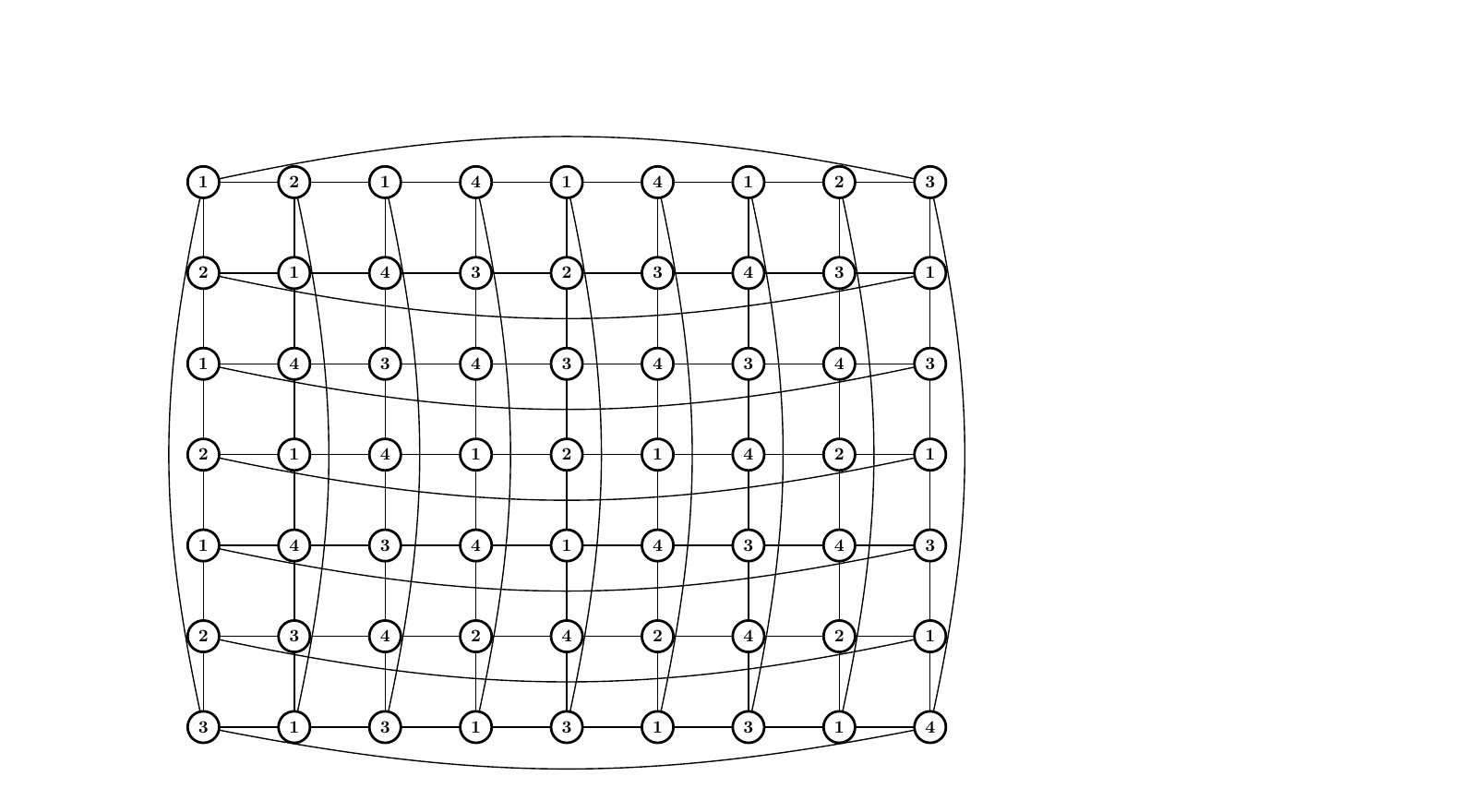}
		$$
	\caption{$C_{7} \square C_9$}
	\label{C7_C9)}
	\end{subfigure}
	
	\vspace{-0.9cm}
	\begin{subfigure}[t]{1\textwidth}
	
		$$
			\includegraphics[trim=3cm 16.5cm 1cm 2.5cm, clip=true, scale=0.8]{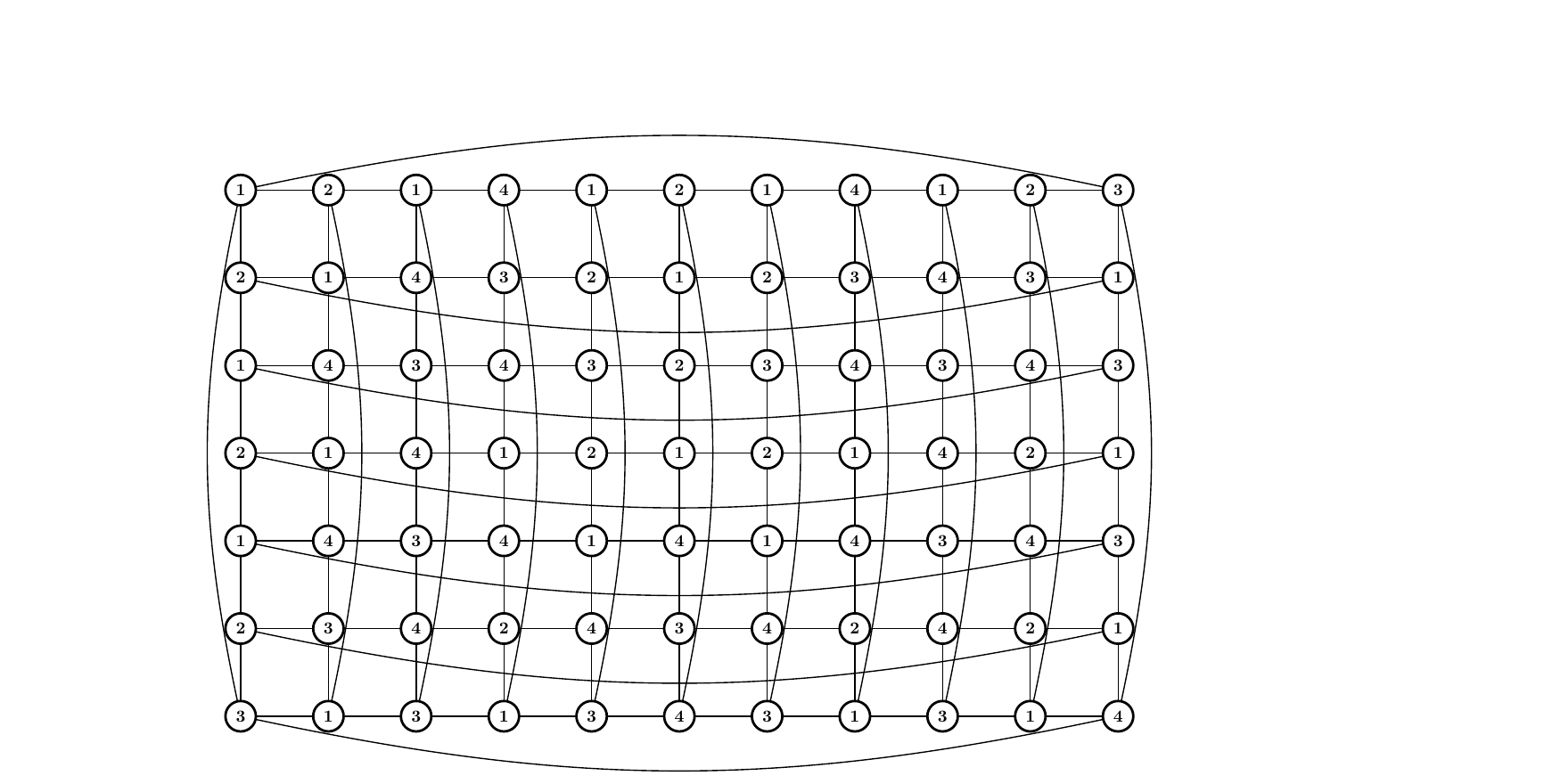}
		$$
	\caption{$C_{7} \square C_{11}$}
	\label{C7_C11}
	\end{subfigure}

	\caption{~ $4$-lid colorings of $C_7 \square C_n$, $n \in \{4,5,7,9,11\}$.}
	\label{fig-C7}
\end{figure}

\begin{figure}[H] 
 \captionsetup[subfigure]{justification=centering}
	\begin{subfigure}[t]{0.4\textwidth}
		$$
			\includegraphics[trim=8cm 15.5cm 1.5cm 3.9cm, clip=true, scale=0.9]{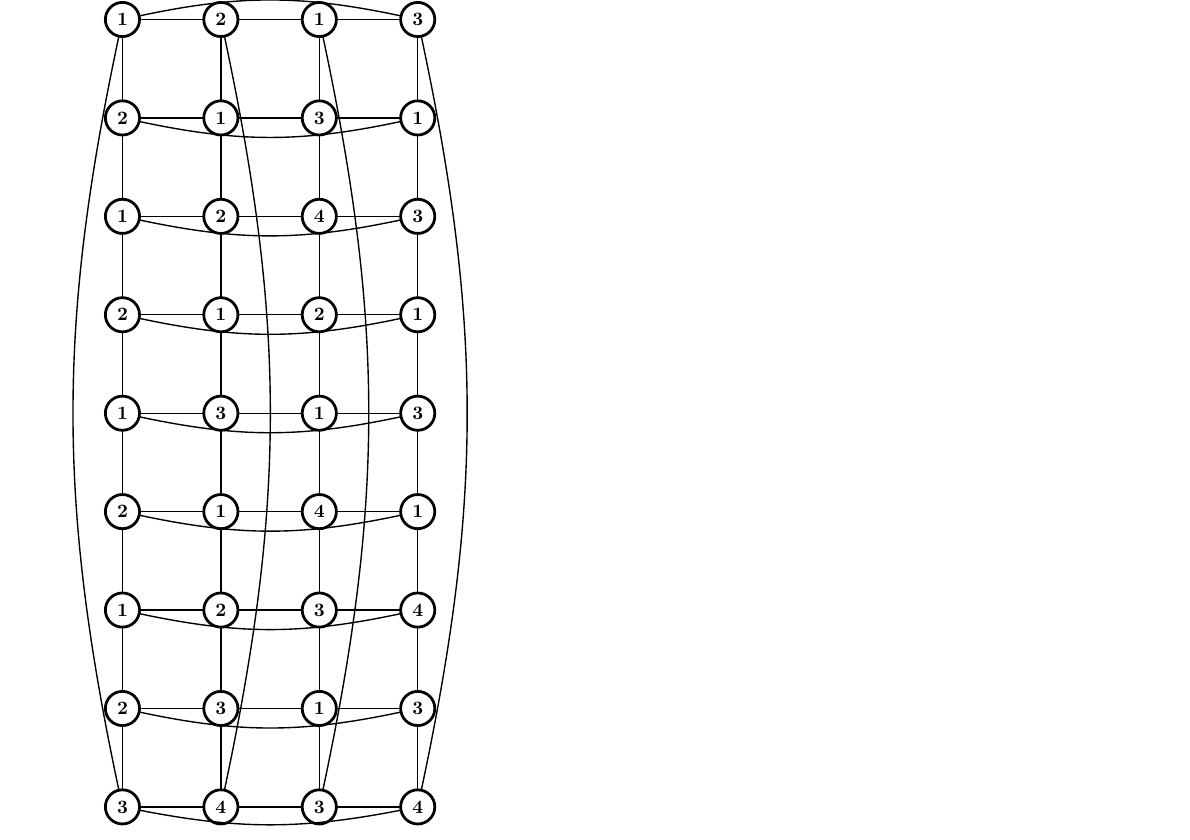}
		$$
	\caption{$C_9 \square C_4$}	
 \label{C9C4}
	\end{subfigure}
	\begin{subfigure}[t]{0.4\textwidth}
	
		$$
			\includegraphics[trim=6.7cm 15.5cm 2cm 3.9cm, clip=true, scale=0.9]{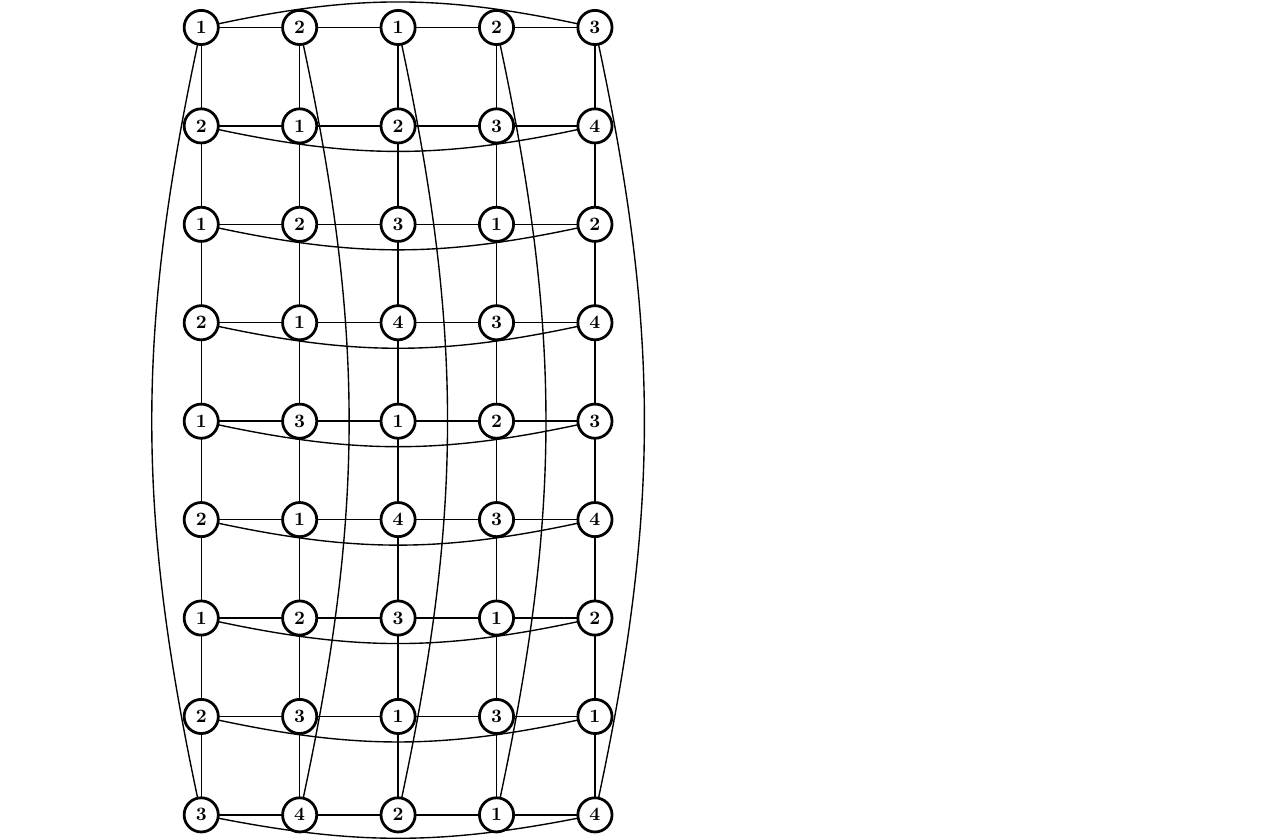}
		$$
	\caption{$C_9 \square C_5$}
	\label{C9C5}
	\end{subfigure}

 \begin{subfigure}[t]{1\textwidth}
	
		$$
			\includegraphics[trim=3cm 14.8cm 1cm 2.5cm, clip=true, scale=0.8]{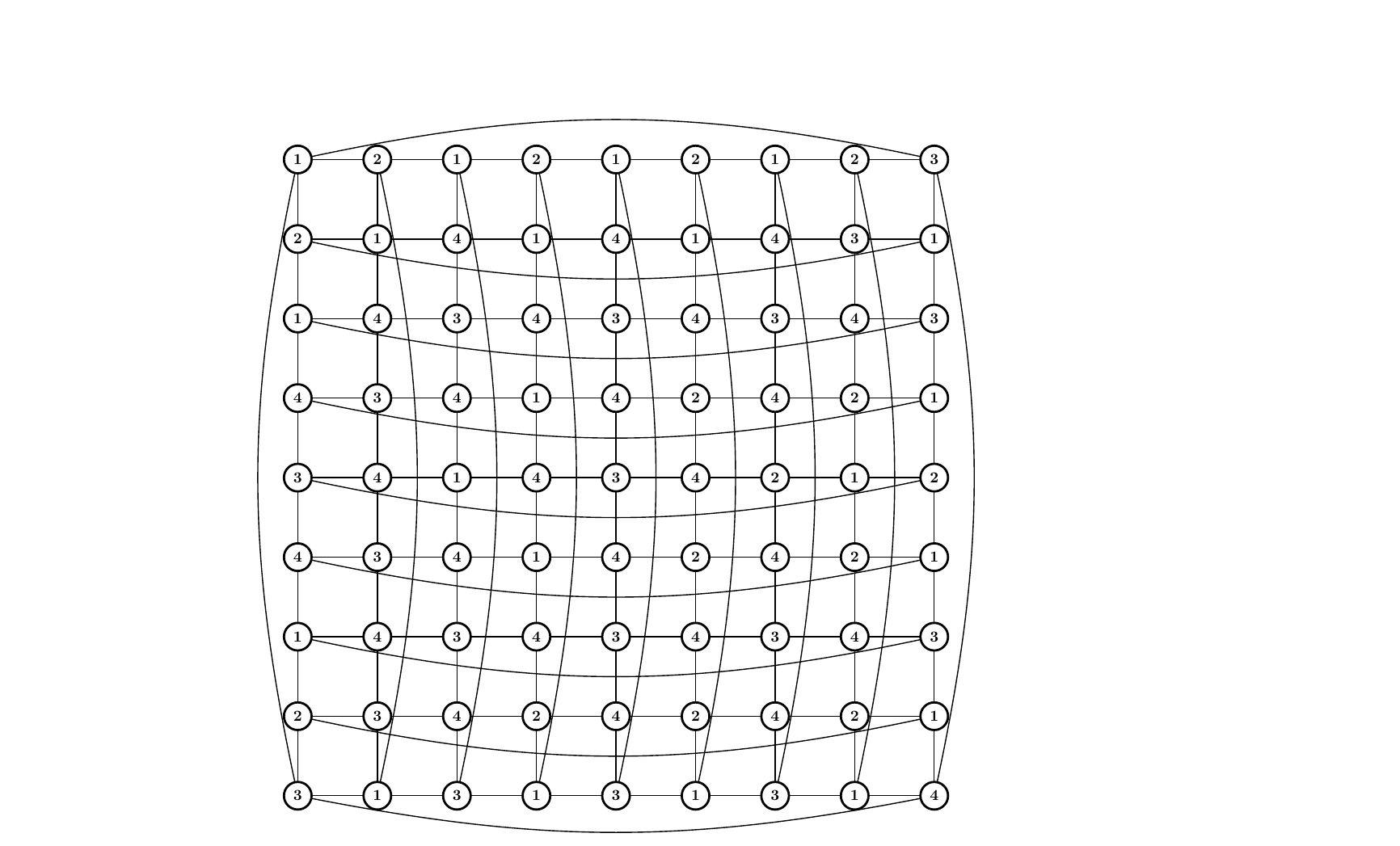}
		$$
	\caption{$C_{9} \square C_9$}
	\label{C9_C9}
	\end{subfigure}
\caption{~ 4-lid-colorings of $C_9 \square C_n$, $n \in \{4,5,9\}$}
 \label{C9CN} 
\end{figure}

\begin{figure}[H] 
 \captionsetup[subfigure]{justification=centering}

\vspace{-1cm}	

	\begin{subfigure}[t]{1\textwidth}
		$$
			\includegraphics[trim=2.5cm 14.8cm 1cm 2.5cm, clip=true, scale=0.9]{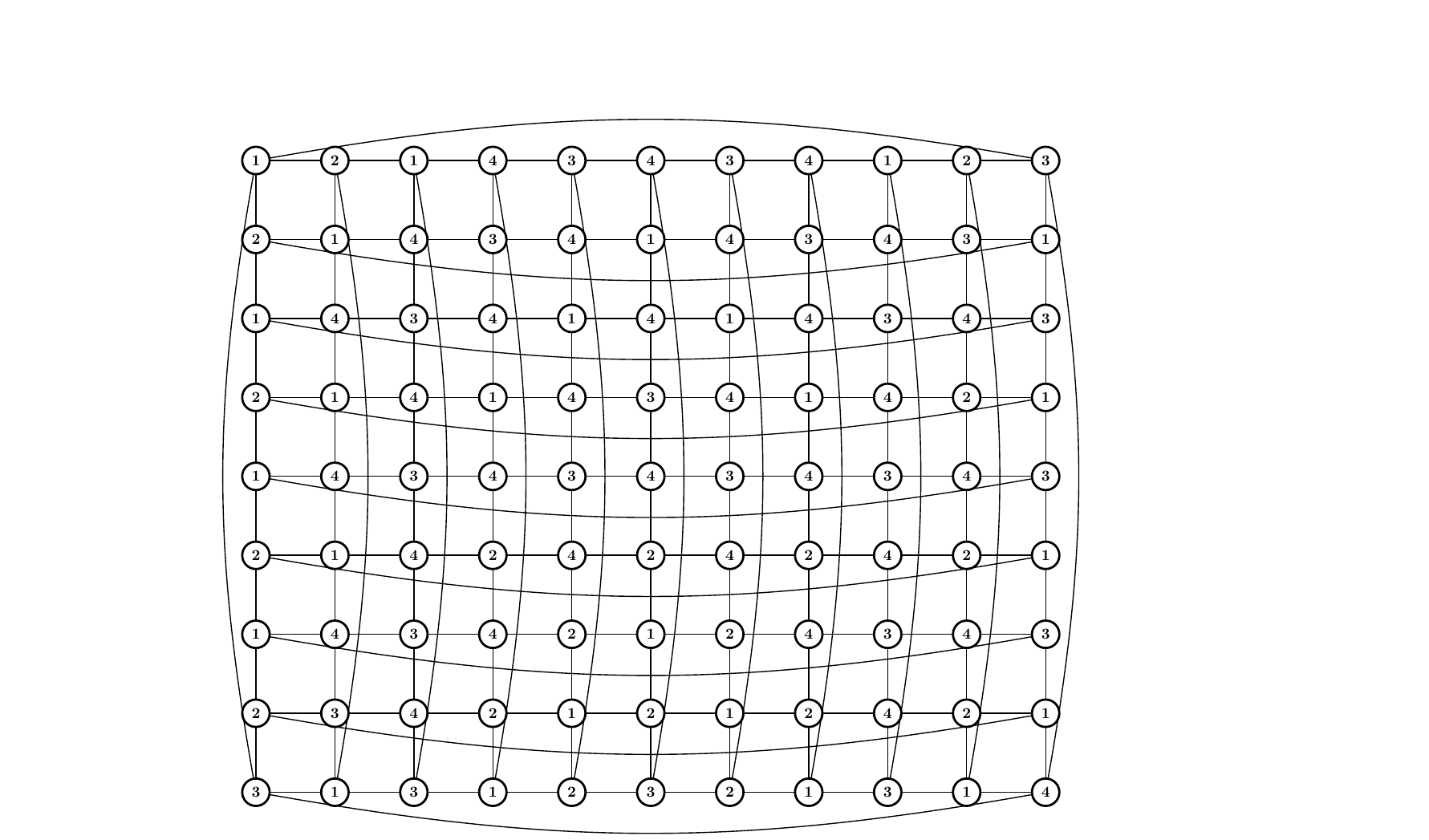}
		$$
	\caption{$C_{9} \square C_{11}$}
	\label{C9_C11)}
	\end{subfigure}
		
	\caption{~A $4$-lid coloring of $C_{9} \square C_{11}$ .}
	\label{fig-C9}
	\end{figure}


\begin{figure}[H] 
 \captionsetup[subfigure]{justification=centering}
 \vspace{-2.5cm}
\begin{subfigure}[t]{0.4\textwidth}
		$$
			\includegraphics[trim=8cm 13cm 1.5cm 3.9cm, clip=true, scale=0.8]{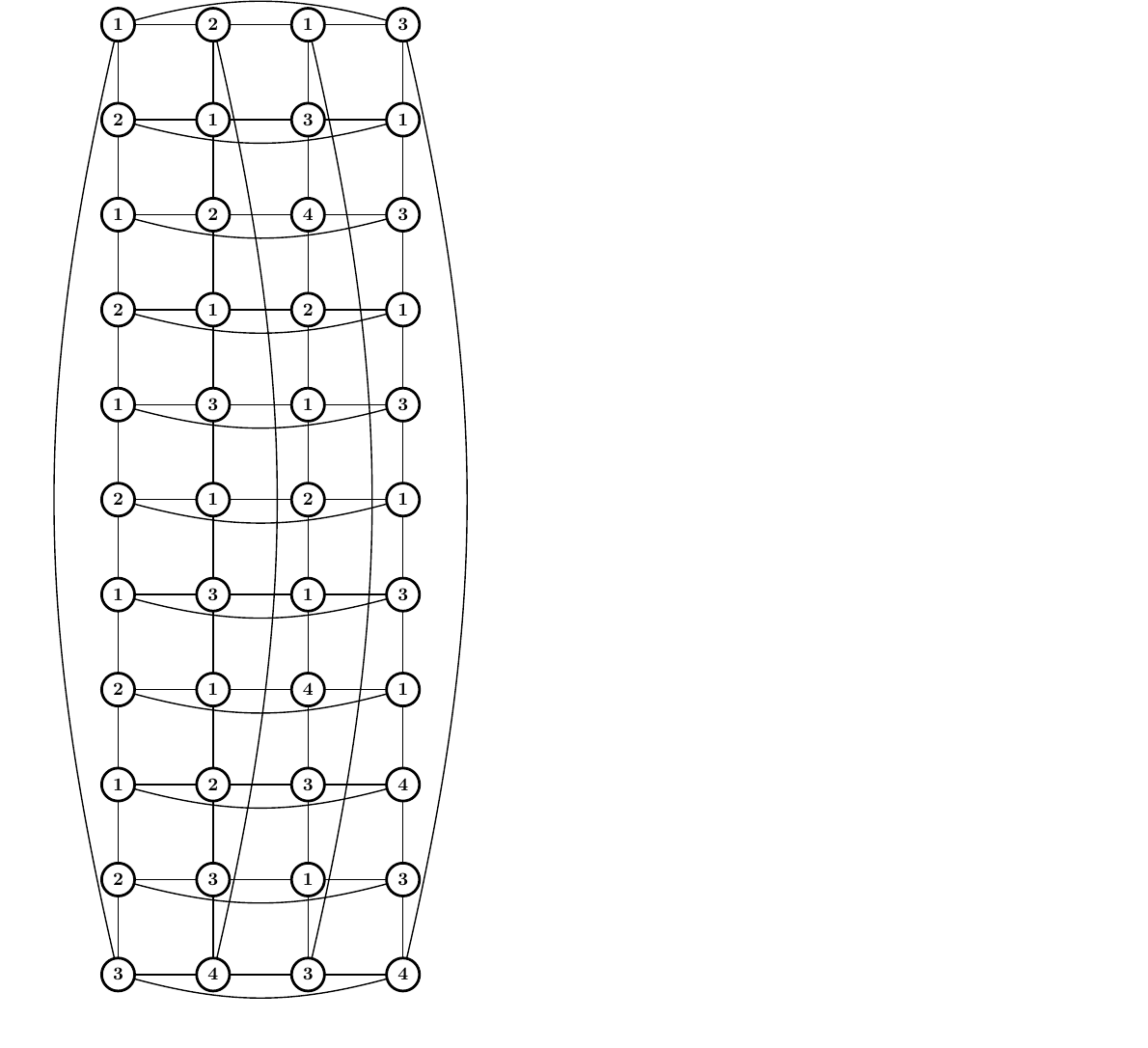}
		$$
	\caption{$C_{11} \square C_4$}	
 \label{C11C4}
	\end{subfigure}
	\begin{subfigure}[t]{0.4\textwidth}
	
		$$
			\includegraphics[trim=6.7cm 13cm 2cm 3.9cm, clip=true, scale=0.8]{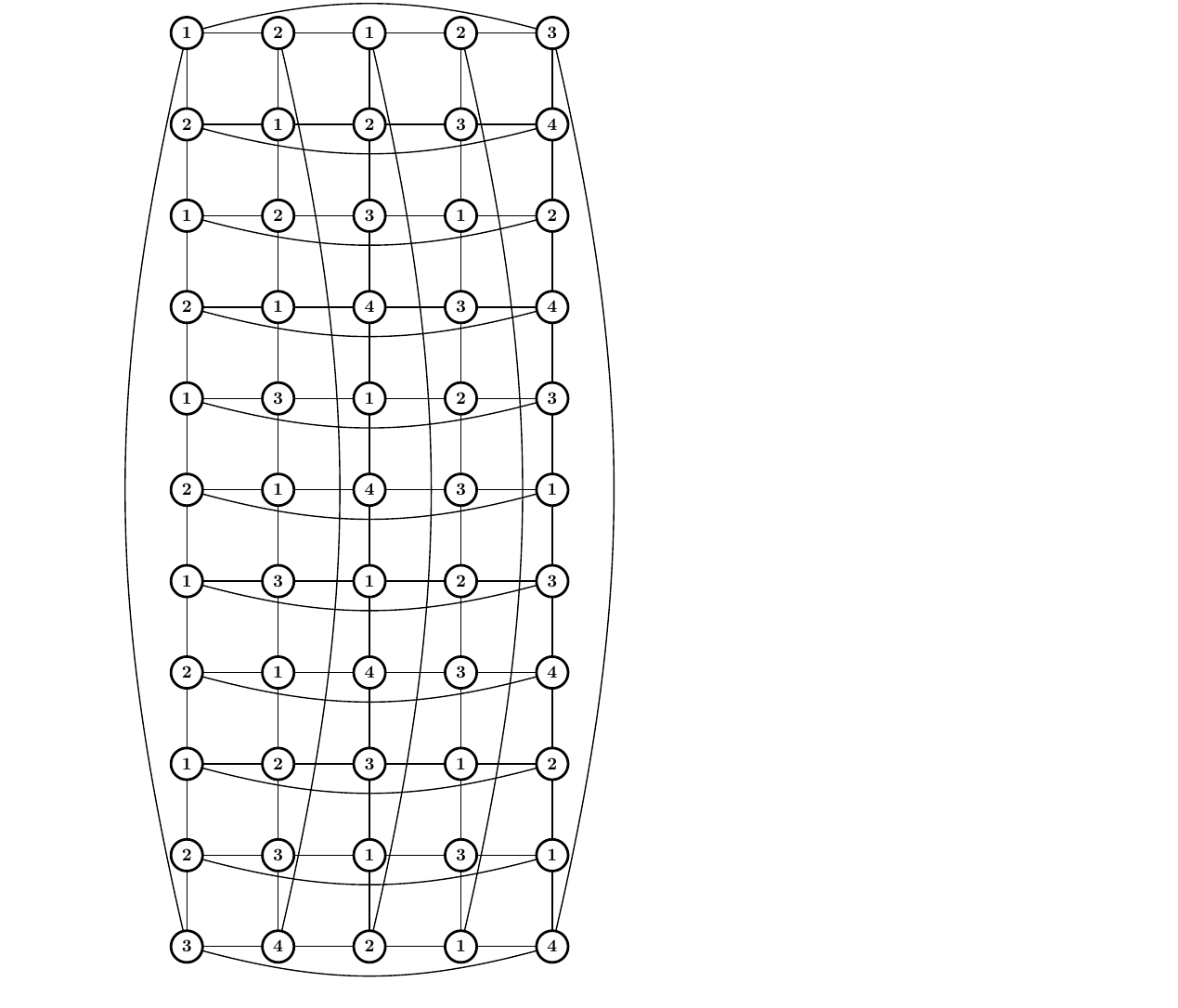}
		$$
	\caption{$C_{11} \square C_5$}
	\label{C11C5}
	\end{subfigure}
\vspace{-0.8cm}
 
	\begin{subfigure}[t]{1\textwidth}
	
		$$
			\includegraphics[trim=4cm 12.5cm 1cm 2.5cm, clip=true, scale=0.8]{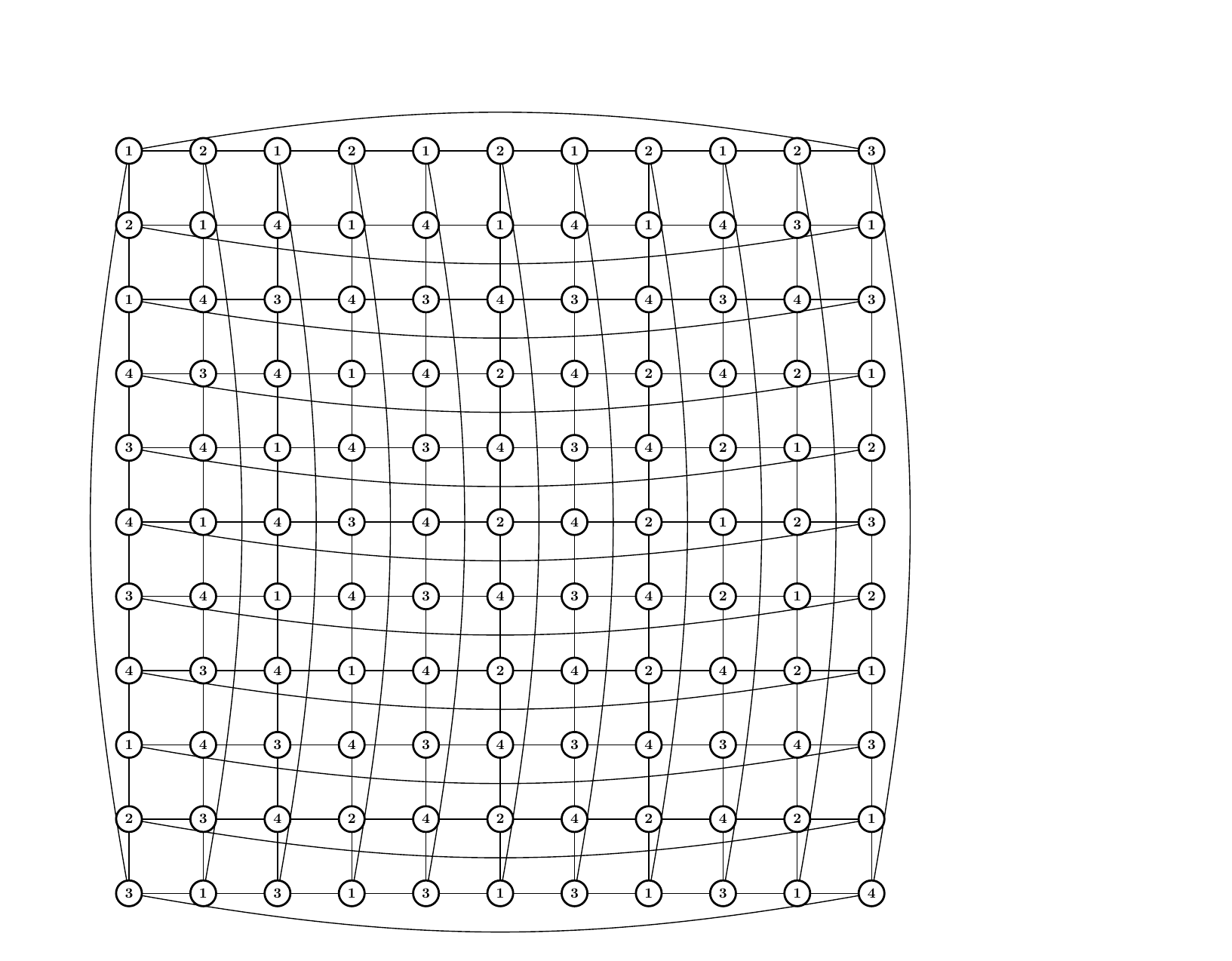}
		$$
	\caption{$C_{11} \square C_{11}$}
	\label{C11_C11}
	\end{subfigure}

	\caption{~ $4$-lid colorings of $C_{11} \square C_n$, $n \in \{4,5,11\}$.}
	\label{fig-C11}
\end{figure}

\subsection
{Figures related to the tensor product of two odd cycles}
\begin{figure}[H] 
 \captionsetup[subfigure]{justification=centering}
	\begin{subfigure}[t]{0.4\textwidth}
    \includegraphics[trim=8cm 19cm 8cm 3cm, clip=true, scale=0.7]{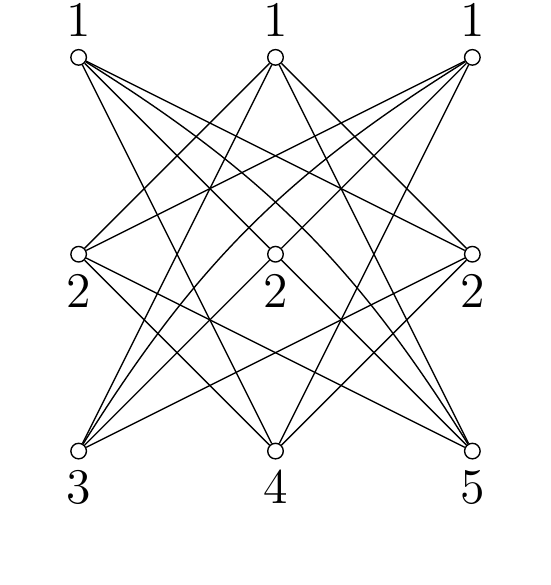}
    \caption{$C_3 \times C_3$}
	\end{subfigure}
\begin{subfigure}[t]{0.55\textwidth}
		\includegraphics[trim=5cm 19cm 5cm 3cm, clip=true, scale=0.7]{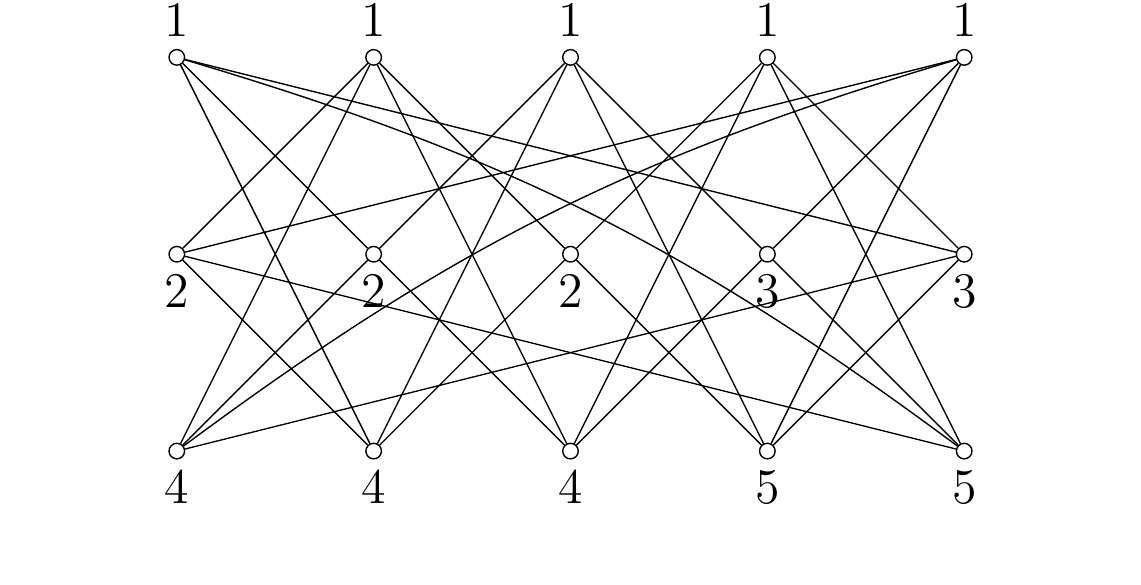}
    \caption{ $C_3 \times C_5$}
    
	\end{subfigure}


 \caption{~ $5$-lid-colorings of $C_3 \times C_3$ and $C_3 \times C_5$}
\label{35}
\end{figure}


\begin{figure}
\includegraphics[trim=2cm 19cm 2cm 1.5cm, clip=true, scale=0.7]{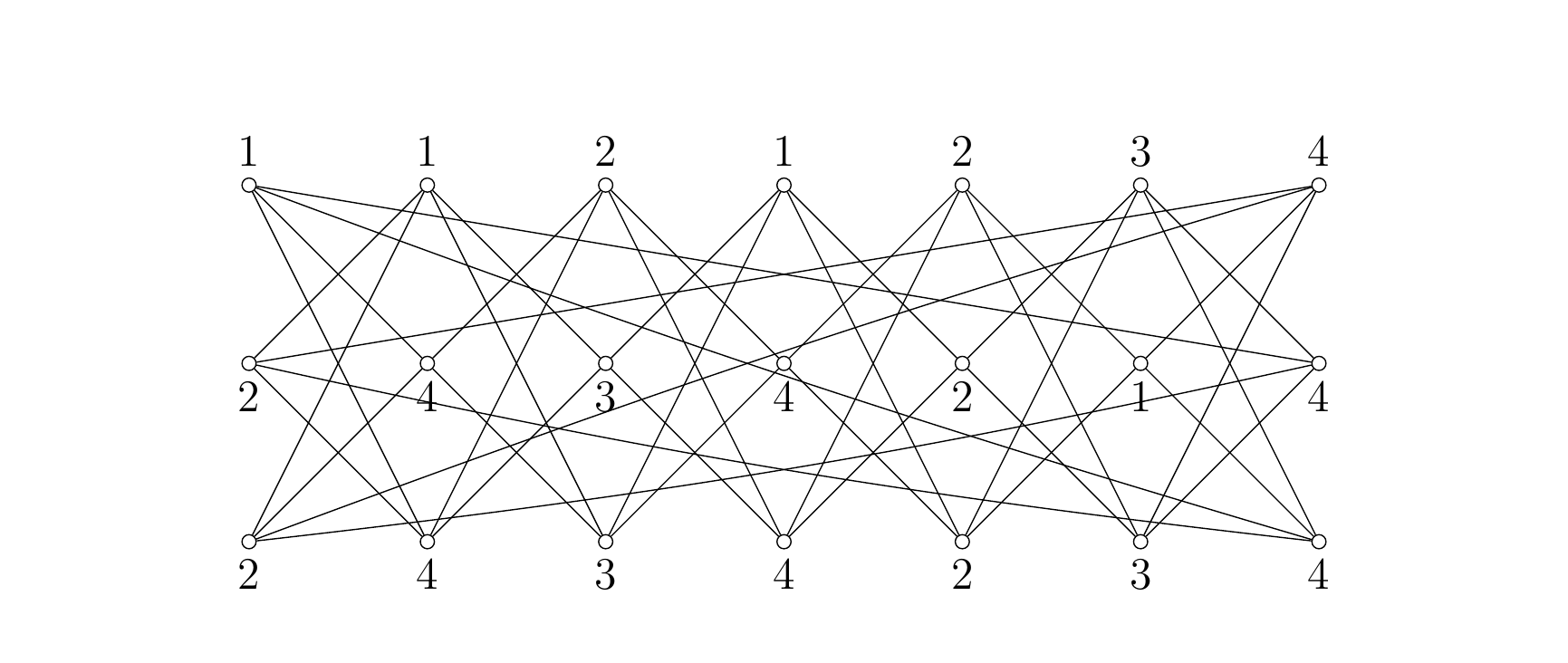}
    \caption{~A $4$-lid-coloring of $C_3 \times C_7$ }
    \label{37}
\end{figure}

\begin{figure}
\includegraphics[trim=2cm 15cm 2cm 1.5cm, clip=true, scale=0.7]{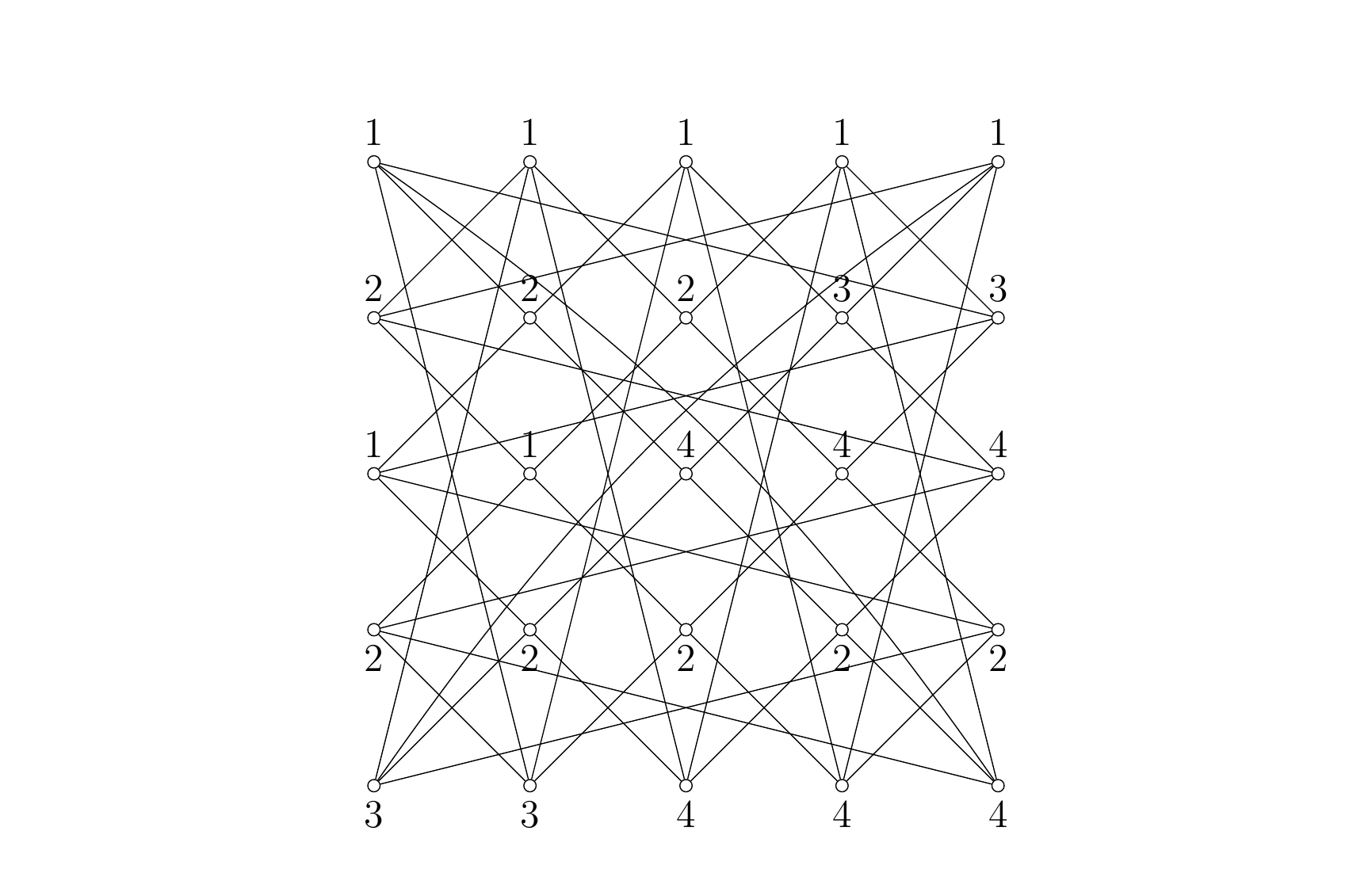}
    \caption{~A $4$-lid-coloring of $C_5 \times C_5$}
    \label{55}	
\end{figure}

\begin{figure}
\includegraphics[trim=2cm 15cm 2cm 1.5cm, clip=true, scale=0.7]{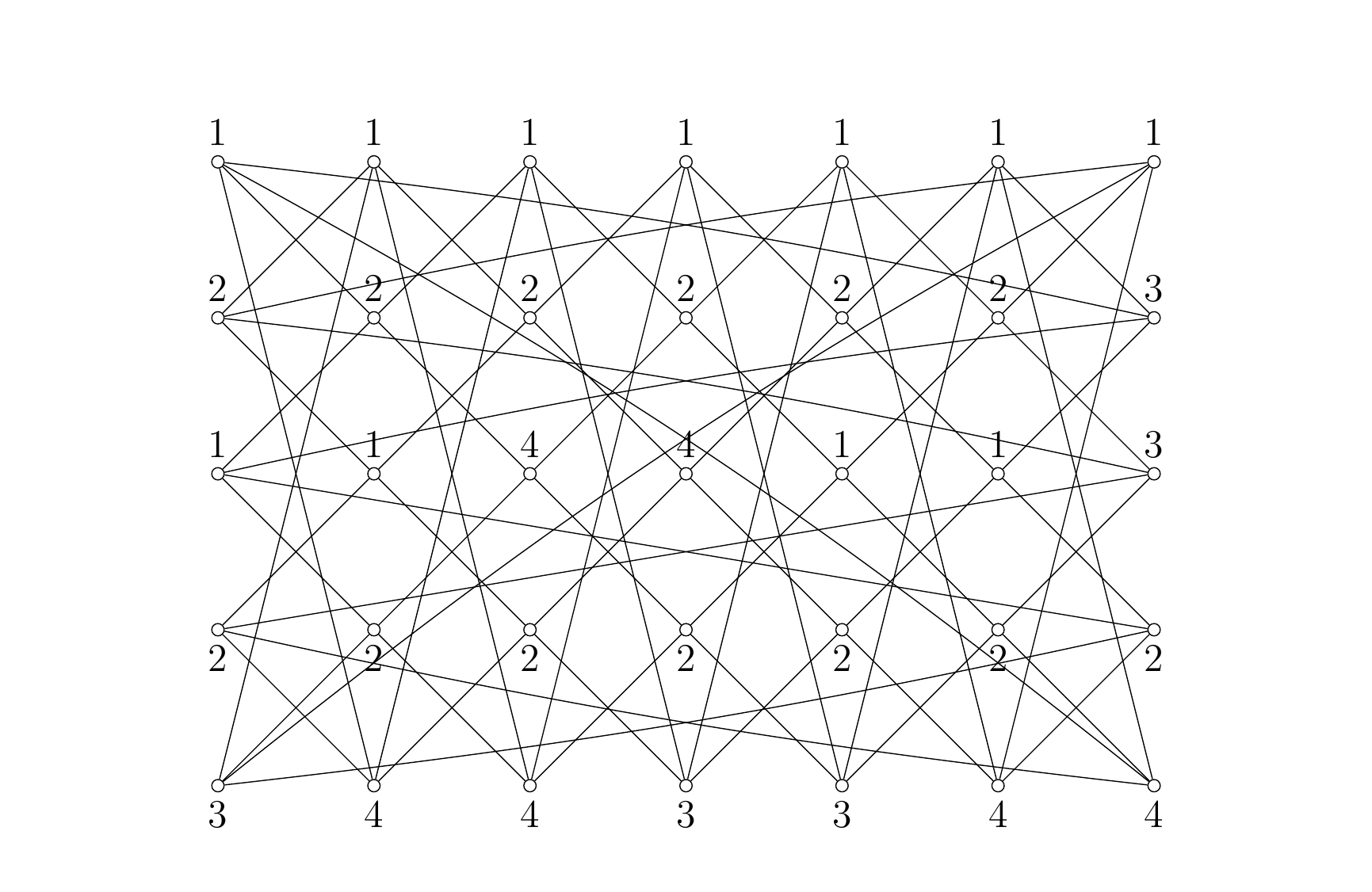}
    \caption{~A $4$-lid-coloring of $C_5 \times C_7$}
    \label{57}
\end{figure}

\begin{figure}
	
		\includegraphics[trim=2cm 11cm 2cm 1.5cm, clip=true, scale=0.7]{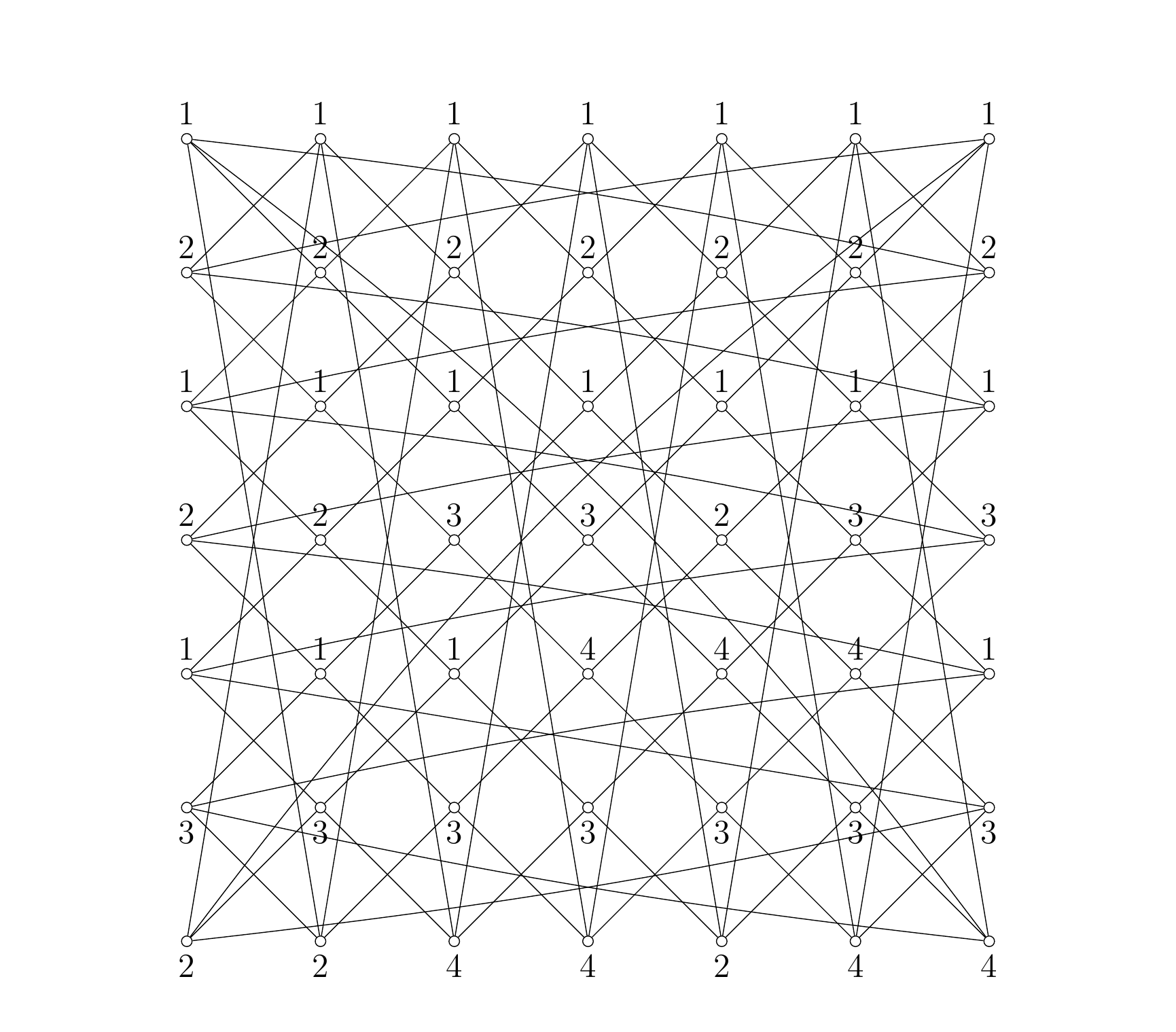}
    \caption{~A $4$-lid-coloring of $C_7 \times C_7$}
    \label{77}
	
\end{figure}
\end{document}